\documentclass[12pt]{amsart}
\usepackage{float} 
\usepackage{amsmath,amsthm,amssymb,enumitem,verbatim,stmaryrd,xcolor,microtype,graphicx,aliascnt,needspace,array,mathrsfs}
\usepackage[T1]{fontenc}
\usepackage[utf8]{inputenc}
\usepackage[english]{babel} 
\usepackage[top=2.95cm,bottom=2.95cm,left=3.5cm,right=3.5cm]{geometry}
\usepackage[bookmarksdepth=2,linktoc=page,colorlinks,linkcolor={red!70!black},citecolor={red!80!black},urlcolor={blue!80!black}]{hyperref}

\usepackage{tikz}
\usepackage{tikz-cd}\usetikzlibrary{arrows}

\usepackage[mathcal]{euscript}                               

\usepackage{mathptmx}                                        
\usepackage[colorinlistoftodos,bordercolor=orange,backgroundcolor=orange!20,linecolor=orange,textsize=footnotesize]{todonotes} \makeatletter \providecommand \@dotsep{5} \def\listtodoname{List of Todos} \def\listoftodos{\@starttoc{tdo}\listtodoname} \makeatother 

\allowdisplaybreaks 

\usepackage{etoolbox}
\makeatletter
\patchcmd{\@startsection}{\@afterindenttrue}{\@afterindentfalse}{}{}             
\patchcmd{\part}{\bfseries}{\bfseries\LARGE}{}{}
\patchcmd{\section}{\scshape}{\bfseries}{}{}\renewcommand{\@secnumfont}{\bfseries} 
\patchcmd{\@settitle}{\uppercasenonmath\@title}{\large}{}{}
\patchcmd{\@setauthors}{\MakeUppercase}{}{}{}
\addto{\captionsenglish}{} 
\addto{\captionsenglish}{} 
\addto{\captionsenglish}{} 
\makeatother

\usepackage{fancyhdr}

\addto\extrasenglish{}  
\theoremstyle{plain}
\newtheorem{thm}{Theorem}[section] 
\newaliascnt{lemma}{thm}\newtheorem{lemma}[lemma]{Lemma}\aliascntresetthe{lemma}
\newaliascnt{cor}{thm}\newtheorem{cor}[cor]{Corollary}\aliascntresetthe{cor}
\newaliascnt{prop}{thm}\newtheorem{prop}[prop]{Proposition}\aliascntresetthe{prop}
\newtheorem{thmA}{Theorem} \newaliascnt{corA}{thmA}\newtheorem{corA}[corA]{Corollary}\aliascntresetthe{corA}
\newaliascnt{lemmaA}{thmA}\aliascntresetthe{lemmaA}
\newtheorem{claim}{Claim}[thm] 
\newtheorem*{thm*}{Theorem}
\newtheorem*{lem*}{Lemma}
\newtheorem*{cor*}{Corollary}

\theoremstyle{definition}
\newaliascnt{df}{thm}\newtheorem{df}[df]{Definition}\aliascntresetthe{df}
\newaliascnt{rem}{thm}\newtheorem{rem}[rem]{Remark}\aliascntresetthe{rem}
\newaliascnt{ex}{thm}\newtheorem{ex}[ex]{Example}\aliascntresetthe{ex}

\newtheorem*{df*}{Definition}
\newtheorem*{ex*}{Example}
\newtheorem*{rem*}{Remark}

\theoremstyle{remark}

\DeclareRobustCommand{\gobblefour}[5]{}    

\DeclareFontFamily{OT1}{pzc}{}                                
\DeclareFontShape{OT1}{pzc}{m}{it}{<-> s * [1.10] pzcmi7t}{}
\DeclareMathAlphabet{\mathpzc}{OT1}{pzc}{m}{it}
\DeclareSymbolFont{sfoperators}{OT1}{bch}{m}{n} \DeclareSymbolFontAlphabet{\mathsf}{sfoperators} \makeatletter\def\operator@font{\mathgroup\symsfoperators}\makeatother 
\DeclareSymbolFont{cmletters}{OML}{cmm}{m}{it}              
\DeclareSymbolFont{cmsymbols}{OMS}{cmsy}{m}{n}
\DeclareSymbolFont{cmlargesymbols}{OMX}{cmex}{m}{n}
\DeclareMathSymbol{\myjmath}{\mathord}{cmletters}{"7C}     \let\jmath\myjmath 
\DeclareMathSymbol{\myamalg}{\mathbin}{cmsymbols}{"71}     
\DeclareMathSymbol{\mycoprod}{\mathop}{cmlargesymbols}{"60}
\DeclareMathSymbol{\myalpha}{\mathord}{cmletters}{"0B}     \let\alpha\myalpha 
\DeclareMathSymbol{\mybeta}{\mathord}{cmletters}{"0C}      \let\beta\mybeta
\DeclareMathSymbol{\mygamma}{\mathord}{cmletters}{"0D}     \let\gamma\mygamma
\DeclareMathSymbol{\mydelta}{\mathord}{cmletters}{"0E}     \let\delta\mydelta
\DeclareMathSymbol{\myepsilon}{\mathord}{cmletters}{"0F}   \let\epsilon\myepsilon
\DeclareMathSymbol{\myzeta}{\mathord}{cmletters}{"10}      \let\zeta\myzeta
\DeclareMathSymbol{\myeta}{\mathord}{cmletters}{"11}       \let\eta\myeta
\DeclareMathSymbol{\mytheta}{\mathord}{cmletters}{"12}     \let\theta\mytheta
\DeclareMathSymbol{\myiota}{\mathord}{cmletters}{"13}      \let\iota\myiota
\DeclareMathSymbol{\mykappa}{\mathord}{cmletters}{"14}     \let\kappa\mykappa
\DeclareMathSymbol{\mylambda}{\mathord}{cmletters}{"15}    \let\lambda\mylambda
\DeclareMathSymbol{\mymu}{\mathord}{cmletters}{"16}        \let\mu\mymu
\DeclareMathSymbol{\mynu}{\mathord}{cmletters}{"17}        \let\nu\mynu
\DeclareMathSymbol{\myxi}{\mathord}{cmletters}{"18}        \let\xi\myxi
\DeclareMathSymbol{\mypi}{\mathord}{cmletters}{"19}        \let\pi\mypi
\DeclareMathSymbol{\myrho}{\mathord}{cmletters}{"1A}       \let\rho\myrho
\DeclareMathSymbol{\mysigma}{\mathord}{cmletters}{"1B}     \let\sigma\mysigma
\DeclareMathSymbol{\mytau}{\mathord}{cmletters}{"1C}       \let\tau\mytau
\DeclareMathSymbol{\myupsilon}{\mathord}{cmletters}{"1D}   \let\upsilon\myupsilon
\DeclareMathSymbol{\myphi}{\mathord}{cmletters}{"1E}       \let\phi\myphi
\DeclareMathSymbol{\mychi}{\mathord}{cmletters}{"1F}       \let\chi\mychi
\DeclareMathSymbol{\mypsi}{\mathord}{cmletters}{"20}       \let\psi\mypsi
\DeclareMathSymbol{\myomega}{\mathord}{cmletters}{"21}     \let\omega\myomega
\DeclareMathSymbol{\myvarepsilon}{\mathord}{cmletters}{"22}\let\varepsilon\myvarepsilon
\DeclareMathSymbol{\myvartheta}{\mathord}{cmletters}{"23}  \let\vartheta\myvartheta
\DeclareMathSymbol{\myvarpi}{\mathord}{cmletters}{"24}     \let\varpi\myvarpi
\DeclareMathSymbol{\myvarrho}{\mathord}{cmletters}{"25}    \let\varrho\myvarrho
\DeclareMathSymbol{\myvarsigma}{\mathord}{cmletters}{"26}  \let\varsigma\myvarsigma
\DeclareMathSymbol{\myvarphi}{\mathord}{cmletters}{"27}    \let\varphi\myvarphi

\DeclareMathOperator{\Hom}{Hom}

\DeclareMathOperator{\cl}{cl}
\DeclareMathOperator{\si}{si}
\newcommand{\del}{\setminus}  

\DeclareMathOperator{\Sym}{Sym}

\newcommand\F{{\mathbb F}}

\newcommand\K{{\mathbb K}}

\newcommand\N{{\mathbb N}}

\renewcommand\P{{\mathbb P}}

\newcommand\R{{\mathbb R}}
\renewcommand\S{{\mathbb S}}
\newcommand\T{{\mathbb T}}

\newcommand\Z{{\mathbb Z}}

\newcommand\cC{{\mathcal C}}

\newcommand\cF{{\mathcal F}}

\newcommand\cH{{\mathcal H}}

\newcommand\cX{{\mathcal X}}

\newcommand\Funpm{{\F_1^\pm}}
\newcommand\id{\textup{id}}

\renewcommand\geq{\geqslant}
\renewcommand\leq{\leqslant}

\newcommand{\gen}[1]{\langle #1 \rangle}

\newcommand{\past}[2]{#1\!\sslash\!#2}

\newcommand{\cross}[5]{\mathchoice{\scalebox{1.3}{$\big[$}\,\raisebox{1pt}{$\begin{matrix}{\scalebox{0.9}{$#1$}}\hspace{-5pt}&{\scalebox{0.9}{$#2$}}\\[-2pt]{\scalebox{0.9}{$#3$}}\hspace{-5pt}&{{\scalebox{0.9}{$#4$}}}\end{matrix}$}\,\scalebox{1.3}{$\big]$}_{#5}}{\big[\begin{smallmatrix}{#1}&{#2}\\{#3}&{#4}\end{smallmatrix}\big]_{#5}}{}{}}   
\renewcommand{\setminus}{\backslash}
\newcommand{\minor}[2]{\setminus #1 / #2}

\renewcommand\emptyset\varnothing

\title{The foundation of generalized parallel connections, 2-sums, and segment-cosegment exchanges of matroids}

\author{Matthew Baker}
\address{\rm Matthew Baker, School of Mathematics, Georgia Institute of Technology, Atlanta, USA}
\email{mbaker@math.gatech.edu}

\author{Oliver Lorscheid}
\address{\rm Oliver Lorscheid, University of Groningen, the Netherlands}
\email{o.lorscheid@rug.nl}

\author{Zach Walsh}
\address{\rm Zach Walsh, Department of Mathematics and Statistics, Auburn University, Auburn, USA}
\email{zwalsh@auburn.edu}

\author{Tianyi Zhang}
\address{\rm Tianyi Zhang, School of Mathematics, Georgia Institute of Technology, Atlanta, USA}
\email{kafuka@gatech.edu}

\hypersetup{pdftitle={Foundation of generalized parallel connections},pdfauthor={Matthew Baker, Oliver Lorscheid, Zach Walsh, and Tianyi Zhang}}

\begin{document}

\begin{abstract}
We show that, under suitable hypotheses, the foundation of a generalized parallel connection of matroids is the relative tensor product of the foundations. Using this result, we show that the foundation of a 2-sum of matroids is the absolute tensor product of the foundations, and that the foundation of a matroid is invariant under segment-cosegment exchange.
\end{abstract}

\thanks{We thank Nathan Bowler and Rudi Pendavingh for numerous helpful suggestions.
M.B. was supported by NSF grant DMS-2154224 and a Simons Fellowship in Mathematics. O.L. was supported by Marie Sk{\l}odowska Curie Fellowship MSCA-IF-101022339.}

\maketitle

\section{Introduction}

{\em Pastures} are algebraic objects that generalize partial fields. In \cite{Baker-Lorscheid25}, Baker and Lorscheid study the \emph{foundation} of a matroid $M$, which is a pasture canonically attached to $M$ that governs the representability of $M$ over arbitrary pastures.
In particular, the foundation $F_M$ determines the set of projective equivalence classes of representations of $M$ over partial fields.
More precisely, for any pasture $P$, the set of (weak) $P$-representations of $M$, modulo rescaling equivalence, is canonically identified with the set of pasture homomorphisms from $F_M$ to $P$.

\medskip

Let $M_1,M_2$ be matroids
with ground sets $E_1$ and $E_2$ respectively. 
If $E_1 \cap E_2 = T$ with $M_1|T = M_2|T$ and $T$ is a modular flat\footnote{A flat $T$ of a matroid $M$ is called {\em modular} if $r(T) + r(F) = r(T \cap F) + r(T \cup F)$ for every flat $F$ of $M$, where $r$ is the rank function of $M$.} in either $M_1$ or $M_2$, then one can define the {\em generalized parallel connection} $P_T(M_1,M_2)$ (cf. \cite[p.441]{Oxley06}) as the matroid on $E = E_1 \cup E_2$ such that $F$ is a flat of $P_T(M_1,M_2)$ if and only if $F\cap E_i$ is a flat of $M_i$ for $i=1,2$.

There are some important constructions in matroid theory which make use of the generalized parallel connection, two of the most important being:

\begin{enumerate}
\item If $M_1$ and $M_2$ are simple and $T = \{ p \}$ is a singleton, then $T$ is automatically a modular flat in both $M_1$ and $M_2$. In this case, we define the {\em 2-sum of $M_1$ and $M_2$ along $p$}, denoted $M_1 \oplus_2 M_2$ (or $M_1 \oplus_p M_2$, if we want to emphasize the dependence on $p$), to be the minor $P_T(M_1,M_2) \setminus T$ of $P_T(M_1,M_2)$.
\item If $T$ is a coindependent triangle (i.e., 3-element circuit) in a matroid $M$, we define the {\em Delta-Wye exchange of $M$ along $T$}, denoted $\Delta_T(M)$, to be the minor $P_T(M,M(K_4)) \setminus T$ of $P_T(M,M(K_4))$, where $T$ is identified with a triangle in $M(K_4)$. 

More generally, if $M$ is a matroid and $X \subseteq E(M)$ is a coindependent set such that $M|X \cong U_{2,n}$ for some $n \ge 2$, one defines the {\em segment-cosegment exchange of $M$ along $X$} to be $P_X(M,\Theta_n) \setminus X$, where $\Theta_n$ is a certain matroid on $2n$ elements defined in \autoref{sec:segmentcosegmentexchange}. When $n = 3$, we have $\Theta_3 \cong M(K_4)$ and the segment-cosegment exchange of $M$ along $X$ coincides with $\Delta_X(M)$.
\end{enumerate}
It is known that a 2-sum of matroids $M_1$ and $M_2$ is representable over a partial field $P$ if and only if $M_1$ and $M_2$ are both representable over $P$ \cite[Corollary 2.4.31]{vanZwam09}. It is also known that if $M$ is a matroid containing a 
coindependent set $X$ such that $M|X \cong U_{2,n}$ for some $n \ge 2$, then $M$ is representable over a partial field $P$ if and only if the segment-cosegment exchange of $M$ along $X$ is representable over $P$ \cite[Corollary 3.6]{Oxley-Semple-Vertigan}. 
In this paper, we generalize these results in two important ways:
\begin{itemize}
    \item We establish bijections between suitable rescaling classes of $F$-representations.
    \item We prove analogous results for representations over arbitrary pastures.
\end{itemize}

\medskip

Our main theorems are as follows:

\begin{thmA}\label{thmA}
Let $M_1$ and $M_2$ be matroids so that $E(M_1) \cap E(M_2) = T$ and $M_1|T = M_2|T$. Suppose that either:
\begin{enumerate}
    \item $T$ is a modular flat of both $M_1$ and $M_2$; or
    \item $T$ is isomorphic to $U_{2,n}$ for some $n \geq 2$ and $M_2$ is isomorphic to $\Theta_n$.
\end{enumerate}
Then the foundation of $P_T(M_1,M_2)$ is isomorphic to $F_{M_1}\otimes_{F_{M_1|T}}F_{M_2}$.
\end{thmA}

Part (1) of \autoref{thmA} is proved in \autoref{sec:GPC}, and part (2) is proved in \autoref{sec:segmentcosegmentexchange}.

In the special case where $T = \varnothing$, we obtain the following corollary (also proved in \cite{Baker-Lorscheid-Zhang25}):

\begin{corA}\label{corA}
The foundation of a direct sum $M_1 \oplus M_2$ is isomorphic to $F_{M_1}\otimes F_{M_2}$.
\end{corA}

\begin{rem*}
When $T$ is a modular flat in $M_2$ but not necessarily in $M_1$, the generalized parallel connection $M = P_T(M_1, M_2)$ is still well-defined, but the identity $F_{P_T(M_1, M_2)} \cong F_{M_1} \otimes_{F_{M_1|T}} F_{M_2}$ does not necessarily hold, even when $r(T) = 2$. We give an example at the end of \autoref{sec:GPC}.

\end{rem*}

In certain situations, the foundations of $P_T(M_1,M_2)$ and $P_T(M_1,M_2) \setminus T$ turn out to be isomorphic. 
The two most important examples are that of 2-sums and segment-cosegment exchanges:


\begin{thmA}\label{thmB}
Let $M_1$ and $M_2$ be matroids on $E_1$ and $E_2$, respectively, so that $E_1 \cap E_2 = \{p\}$ and $p$ is not a loop or a coloop in $M_1$ or $M_2$.
Then the foundation of the 2-sum $M_1 \oplus_p M_2$ is isomorphic to $F_{M_1}\otimes F_{M_2}$.
\end{thmA}

\begin{thmA}\label{thmC}
Let $M$ be a matroid and let $X \subseteq E(M)$ be a coindependent set such that $M|X \cong U_{2,n}$ for some $n \ge 2$.
Then the foundation of the segment-cosegment exchange of $M$ along $X$ is isomorphic to $F_M$.
\end{thmA}

A proof of \autoref{thmB} is given in \autoref{sec:foundation2sum}. \autoref{thmB} implies, in particular,
that (under the hypotheses of \autoref{thmB}) for every partial field $P$ there is a bijection between rescaling equivalence classes of $P$-representations of $M_1 \oplus_p M_2$ and pairs of rescaling equivalence classes of $P$-representations of $M_1$ and $M_2$.
To the best of our knowledge, even this particular consequence of \autoref{thmB} is new.

\autoref{thmC} is proved in \autoref{sec:segmentcosegmentexchange}. It generalizes a result of Oxley--Semple--Vertigan \cite[Corollary 3.6]{Oxley-Semple-Vertigan} which says that, under the hypotheses of \autoref{thmC}, for every partial field $P$ there is a bijection between rescaling equivalence classes of $P$-representations of $M$ and rescaling equivalence classes of $P$-representations of the segment-cosegment exchange of $M$ along $X$.

The proof of \autoref{thmB} relies on part (1) of \autoref{thmA}, and the proof of \autoref{thmC} relies on part (2) of \autoref{thmA}.

\begin{rem*}
The foundation of $M' = P_T(M_1,M_2) \setminus T$ is not in general isomorphic to the foundation of $M = P_T(M_1,M_2)$, even when $E(M_1)$ and $E(M_2)$ are both modular in $M$. For example, if $N$ is any non-regular matroid on $E$ and $M_i = N\oplus e_i$ with $e_i \not\in E$ for $i = 1,2$, then $E(N)$ is a modular flat of both $M_1$ and $M_2$, so by \autoref{thmA} we have $F_M \cong F_{M_1}\otimes_{F_N} F_{M_2}$.
However, $F_{M'} = F_{e_1 \oplus e_2} \cong \F_1^\pm$, whereas $F_{M_1}\otimes_{F_N} F_{M_2} \cong F_N \not\cong \F_1^\pm$.   
\end{rem*}

Since the universal partial field of a matroid can be computed from its foundation (cf.~\cite[Lemma 7.48]{Baker-Lorscheid21b} and \autoref{sec:Pvz} below), 
\autoref{thmC} implies in particular an affirmative solution to Conjecture 3.4.4 in Stefan van Zwam's thesis \cite{vanZwam09} (see \autoref{sec:Pvz} for a proof):


\begin{corA}\label{corC}
Let $M$ be a matroid and let $X \subseteq E(M)$ be a coindependent set such that $M|X \cong U_{2,n}$ for some $n \ge 2$, and assume that $M$ is representable over some partial field.
Then the universal partial field of the segment-cosegment exchange of $M$ along $X$ is isomorphic to the universal partial field of $M$.
\end{corA}

\autoref{thmC} also has the following consequence for excluded minors (which is proved in \cite[Theorem 1.1]{Oxley-Semple-Vertigan} in the special case where $P$ is a partial field); for a proof, see Corollary~\ref{cor: foundations for segment-cosegment exchange}.

\begin{corA} \label{corD}
Let $P$ be a pasture, and let $M$ be an excluded minor for representability over $P$. Then every segment-cosegment exchange of $M$ is also an excluded minor for representability over $P$.
\end{corA}

By applying Theorems~\ref{thmB} and \ref{thmC} to $\Hom(F_M,P)$ for certain pastures $P$, we obtain some interesting consequences for $P$-representability. These consequences are already known when $P$ is a partial field, but when $P=\S$ (the sign hyperfield) or $\T$ (the tropical hyperfield), we obtain what appear to be new results. In order to state these corollaries precisely, we recall the following definitions: 

\begin{df*}
(1) A matroid $M$ is called {\em orientable} if $\Hom(F_M,\S)$ is non-empty. (This is equivalent to the usual notion of orientability, cf.~\cite[Example 3.33]{Baker-Bowler19}.)

(2) A matroid $M$ is called {\em rigid} if $\Hom(F_M,\T)$ has more than one element. (This is equivalent to the condition that the base polytope of $M$ has no non-trivial regular matroid polytope subdivision, cf.~\cite[Proposition B.1]{Baker-Lorscheid24}.) Equivalently, $M$ is rigid if and only if every homomorphism $F_M \to \T$ factors through the canonical inclusion $\K \to \T$, where $\K$ is the Krasner hyperfield.
\end{df*}

We have the following straightforward corollaries of Theorems~\ref{thmB} and \ref{thmC}, respectively.

\begin{corA}\label{corF}
Let $M_1$ and $M_2$ be matroids on $E_1$ and $E_2$, respectively, so that $E_1 \cap E_2 = \{p\}$ and $p$ is not a loop or a coloop of $M_1$ or $M_2$.
Then the 2-sum $M_1 \oplus_p M_2$ is orientable (resp. rigid) if and only if $M_1$ and $M_2$ are both orientable (resp. rigid). 
\end{corA}

\begin{proof}
Let $N=M_1\oplus_pM_2$ and $F_{M_1}$, $F_{M_2}$ and $F_N$ be the foundations of $M_1$, $M_2$ and $N$, respectively. Then $M_1$ and $M_2$ are both orientable if and only if both $\Hom(F_{M_1},\S)$ and $\Hom(F_{M_2},\S)$ are non-empty. 
By the universal property of the tensor product in the category of pastures \cite[Lemma 2.7]{Baker-Lorscheid25},
there is a canonical bijection
\[
 \Hom(F_{M_1},\S) \times\Hom(F_{M_2},\S) \ = \ \Hom(F_{M_1}\otimes F_{M_2},\S).
\]
Moreover, by \autoref{thmB} we have $F_N \cong F_{M_1}\otimes F_{M_2}$.
Thus $M_1$ and $M_2$ are both orientable if and only if 
\[
\Hom(F_{M_1},\S) \times\Hom(F_{M_2},\S) \ = \ \Hom(F_{M_1}\otimes F_{M_2},\S) \ = \ \Hom(F_N,\S)
\]
is non-empty. This is, in turn, equivalent to $N=M_1\oplus_pM_2$ being orientable. 

The claim for rigid matroids follows from the same proof, replacing ``orientable'' by ``rigid'', non-empty by singleton, and $\S$ by $\T$ throughout.
\end{proof}

\begin{corA}
Let $M$ be a matroid and let $X \subseteq E(M)$ be a coindependent set such that $M|X \cong U_{2,n}$ for some $n \ge 2$.
Then the segment-cosegment exchange of $M$ along $X$ is orientable (resp. rigid) if and only if $M$ is orientable (resp. rigid).
\end{corA}

\begin{proof}
 By \autoref{thmC}, the foundation of the segment-cosegment exchange of $M$ along $X$ is isomorphic to the foundation of $M$. Since the notions of orientability and rigidity for a matroid $M$ depend only on the foundation of $M$, the claim follows.
\end{proof}

\section{Background on foundations and representations of matroids over pastures}

In this section, we recall some background material from \cite{Baker-Lorscheid25} which will be used throughout this paper. We also discuss some preliminary facts about generalized parallel connections which we will need.

\subsection{Pastures}
Pastures are a generalization of the notion of field in which we still have a multiplicative abelian group $G$, an absorbing element $0$, and an ``additive structure'', but we relax the requirement that the additive structure come from a binary operation.

By a \emph{pointed monoid} we mean a multiplicatively written commutative monoid $P$ with an element $0$ that satisfies $0\cdot a=0$ for all $a\in P$. We denote the unit of $P$ by $1$ and write $P^\times$ for the group of invertible elements in $P$. 
We denote by $\Sym_3(P)$ all elements of the form $a+b+c$ in the monoid semiring $\N[P]$, where $a,b,c\in P$.

\begin{df}
 A \emph{pasture} is a pointed monoid $P$, 
 together with a subset $N_P$ of $\Sym_3(P)$, such that $a \in P^\times$ for all nonzero $a \in P$ 
 and for all $a,b,c,d\in P$ we have:
 \begin{enumerate}[label={(P\arabic*)}]
  \item\label{P1} $a+0+0\in N_P$ if and only if $a=0$,
  \item\label{P2} if $a+b+c\in N_P$, then $ad+bd+cd$ is in $N_P$,
  \item\label{P3} there is a unique element $\epsilon\in P^\times$ such that $1+\epsilon+0\in N_P$.
 \end{enumerate}
\end{df}
We call $N_P$ the \emph{nullset of $P$}, and say that \emph{$a+b+c$ is null}, and write symbolically $a+b+c=0$, if $a+b+c\in N_P$. 
The element $\epsilon$ plays the role of an additive inverse of $1$, and the relations $a+b+c=0$ express that certain sums of elements are zero, even though the multiplicative monoid $P$ does not carry an addition. For this reason, we will write frequently $-a$ for $\epsilon a$ and $a-b$ for $a+\epsilon b$. In particular, we have $\epsilon=-1$. 

A {\em morphism} of pastures is a multiplicative map
$f : P \to P'$ of monoids such that $f(0)=0$, $f(1)=1$ and $f(a)+f(b)+f(c)=0$ in $P'$ whenever $a+b+c=0$ in $P$.

\subsubsection{Examples}
Every field $F$ can be considered as a pasture whose underlying monoid equals that of $F$ and whose nullset is $N_F=\{a+b+c\mid a+b+c=0\text{ in }F\}$. 

Other examples of interest are the following:
\begin{enumerate}
 \item The \emph{regular partial field} is the pointed monoid $\Funpm=\{0,1,-1\}$ (with the obvious multiplication) together with the nullset $N_{\Funpm}=\{0,\ 1-1\}$.
 \item The \emph{Krasner hyperfield} is the pointed monoid $\K=\{0,1\}$ (with the obvious multiplication) together with the nullset $N_\K=\{0,\ 1+1,\ 1+1+1\}$.
 \item The \emph{sign hyperfield} is the pointed monoid $\S=\{0,1,-1\}$ (with the obvious multiplication) together with the nullset $N_\S=\{0,\ 1-1,\ 1+1-1,\ 1-1-1\}$.
 \item The \emph{tropical hyperfield} is the pointed monoid $\T=\R_{\geq0}$ (with the obvious multiplication) together with the nullset $N_\T=\{a+b+b\mid a\leq b\}$.
\end{enumerate}

\subsubsection{Tensor products}
The category of pastures contains all limits and colimits. For example, $\Funpm$ is initial and $\K$ is terminal, i.e., for every pasture $P$, there are unique morphisms $\Funpm\to P$ and $P\to\K$.

The categorical construction that is most essential to this paper is the tensor product (or push-out). Namely given pasture morphisms $\alpha_1:P_0\to P_1$ and $\alpha_2:P_0\to P_2$, there is a pasture $P_1\otimes_{P_0}P_2$ together with morphisms $\iota_1:P_1\to P_1\otimes_{P_0}P_2$ and $\iota_2:P_2\to P_1\otimes_{P_0}P_2$ such that $\iota_1\circ\alpha_1=\iota_2\circ\alpha_2$ that is universal in the sense that for every pair of pasture morphisms $f_1:P_1\to Q$ and $f_2:P_2\to Q$ with $f_1\circ\alpha_1=f_2\circ\alpha_2$, there is a unique pasture morphism $f:P_1\otimes_{P_0}P_2\to Q$ such that $f_1=f\circ\iota_1$ and $f_2=f\circ\iota_2$. In other words, there is a canonical bijection
\[
 \Hom(P_1\otimes_{P_0}P_2,\ Q) \ \longrightarrow \ \Hom(P_1,\ Q) \ \times_{\Hom(P_0,Q)} \ \Hom(P_2,\ Q)
\]
that is functorial in $Q$. This property determines $P_1\otimes_{P_0}P_2$, together with $\iota_1$ and $\iota_2$, uniquely up to unique isomorphism. For the construction of $P_1\otimes_{P_0}P_2$, we refer the reader to \cite{Creech21}.

\subsection{Representations of matroids over pastures}

Let $P$ be a pasture and let $M$ be a matroid on the finite set $E$. 
There are various ``cryptomorphic'' descriptions of weak $P$-matroids, for example in terms of ``weak $P$-circuits'', cf. \cite{Baker-Bowler19}. 
For the purposes of the present paper, however, it will be more convenient to define weak $P$-matroids in terms of modular systems of hyperplane functions, as in \cite[Section 2.3]{Baker-Lorscheid25}.
The point here is that generalized parallel connections are defined in terms of flats, so we have easier access to the hyperplanes of a generalized parallel connection than to the bases or circuits.

\begin{df} \label{df:P-hyperplanes}
Let $\mathscr H$ be the set of hyperplanes of $M$.
\begin{enumerate}
\item Given $H \in \mathscr H$, we say that $f_H : E \to P$ is a {\em $P$-hyperplane function} for $H$ if $f_H(e)=0$ if and only if $e \in H$.
\item A triple of hyperplanes $(H_1,H_2,H_3) \in \mathscr H^3$ is \emph{modular} if $F=H_1\cap H_2\cap H_3$ is a flat of corank $2$ such that $F=H_i\cap H_j$ for all distinct $i,j\in\{1,2,3\}$.
\item A {\em modular system} of $P$-hyperplane functions for $M$ is a collection of $P$-hyperplane functions $f_H : E \to P$, one for each $H \in \mathscr H$, such that
whenever $H_1,H_2,H_3$ is a modular triple of hyperplanes in $\mathscr H$, the corresponding functions $f_{H_i}$ are linearly dependent, i.e., there exist constants $c_1,c_2,c_3$ in $P$, not all zero, such that 
\[
c_1 f_{H_1}(e) + c_2 f_{H_2}(e) + c_3 f_{H_3}(e) = 0
\]
for all $e \in E$.
\end{enumerate}
\end{df}

\begin{df} \label{df:representation}
\begin{enumerate}
    \item A \emph{$P$-representation} of $M$ is a modular system of $P$-hyperplane functions for $M$.
    \item Two $P$-representations $\{ f_H \}$ and $\{ f'_H \}$ of $M$ are \emph{isomorphic} if there is a function $H \mapsto c_H$ from $\mathscr H$ to $P^\times$ such that $f'_H(e)=c_H f_H(e)$ for all $e \in E$ and $H \in \mathscr H$.
    \item Two $P$-representations $\{ f_H \}$ and $\{ f'_H \}$ of $M$ are \emph{rescaling equivalent}
    if there are functions $H \mapsto c_H$ from $\mathscr H$ to $P^\times$ and $e \mapsto c_e$ from $E$ to $P^\times$ such that $f'_H(e)=c_H c_e f_H(e)$ for all $e \in E$ and $H \in \mathscr H$.
\end{enumerate}
\end{df}

When $P$ is a partial field, a rescaling equivalence class of $P$-representations of $M$ is the same thing as a projective equivalence class of $P$-representations of $M$ in the sense of \cite{Pendavingh-vanZwam10b}.
When $P$ is a field, the equivalence between the notion of representability provided in \autoref{df:representation} and the usual notion of matroid representability over a field is precisely the content of ``Tutte's representation theorem'', cf.~\cite[Theorem 5.1]{Tutte65}. 

\begin{rem}\label{rem: rescaling equivalence for hyperplane functions and Grassmann-Pluecker functions}
 The notion of rescaling classes of $P$-representations given by \autoref{df:representation} is compatible with the notion of rescaling classes of $P$-representations given in \cite[Section 1.4.7]{Baker-Lorscheid21b}. Indeed, by \cite[Thm.\ 2.16]{Baker-Lorscheid25}, for every modular system $\{f_H\}$ of hyperplane functions for $M$ in $P$, there is a weak Grassmann-Pl\"ucker function $\Delta:E^r\to P$ representing $M$ such that
 \[
  \frac{f_H(e)}{f_H(e')} \ = \ \frac{\Delta(e,e_2,\dotsc,e_r)}{\Delta(e',e_2,\dotsc,e_r)}
 \]
 for every $H\in\mathscr H$ and all $e,e',e_2,\dotsc,e_r\in E$ such that $\{e_2,\dotsc,e_r\}$ spans $H$ and $\{e',e_2,\dotsc,e_r\}$ is a basis of $M$. The weak Grassmann-Pl\"ucker function $\Delta$ is uniquely determined up to a constant $c\in P^\times$, and two modular systems of hyperplane functions $\{f_H\}$ and $\{f'_{H}\}$ correspond to the same weak Grassmann-Pl\"ucker function $\Delta:E^r\to P$ (up to a constant) if and only if they are isomorphic.
 
 Two weak Grassmann-Pl\"ucker functions $\Delta$ and $\Delta'$ are rescaling equivalent if there are a constant $c\in P^\times$ and a function $e\mapsto c_e$ from $E\to P^\times$ such that
 \[
  \Delta'(e_1,\dotsc,e_r) \ = \ c\cdot c_{e_1} \dotsb c_{e_r} \cdot \Delta(e_1,\dotsc,e_r).
 \]
 Consequently, we have
 \[
  \frac{\Delta'(e,e_2,\dotsc,e_r)}{\Delta'(e',e_2,\dotsc,e_r)} \ = \ \frac{c_e\cdot\Delta(e,e_2,\dotsc,e_r)}{c_{e'}\cdot\Delta(e',e_2,\dotsc,e_r)} \ = \ \frac{c_e\cdot f_H(e)}{c_{e'}\cdot f_H(e')}, 
 \]
 where $H\in\mathscr H$ and $e,e',e_2,\dotsc,e_r\in E$ are as before. This establishes a bijection
 \[
  \bigg\{\begin{array}{c} \text{rescaling classes of weak Grassmann-} \\ \text{Pl\"ucker functions for $M$ in $P$}\end{array}\bigg\} \ 
  \overset{\raise-.5ex\hbox{$\sim$}}{\longrightarrow}
  \ \bigg\{\begin{array}{c} \text{rescaling classes of modular systems} \\ \text{of hyperplane functions for $M$ in $P$}\end{array}\bigg\}.
 \]

\end{rem}

\subsection{The universal pasture and the foundation}

Let $\cX^I_{M}(P)$ (resp. $\cX^R_{M}(P)$) be the set of isomorphism classes (resp. rescaling equivalence classes) of $P$-representations of $M$.
It is shown in \cite{Baker-Lorscheid25} that the functors $\cX^I_M$ and $\cX^R_M$ are representable by the universal pasture $\tilde{F}_M$ and the foundation $F_M$, respectively.
This is equivalent to the fact that $\cX^I_M(P) = \Hom(\tilde{F}_M,P)$ (resp. $\cX^R_M(P) = \Hom(F_M,P)$) functorially in $P$.

In particular, in order to show that some pasture $F'$ is isomorphic to the foundation of $M$, it is equivalent to show that for every morphism of pastures $P \to P'$ there is a commutative diagram
\[
 \begin{tikzcd}
  \Hom(F',P) \arrow{r}{\cong}\arrow[d]
  & \cX^R_M(P) \arrow[d] \\
   \Hom(F',P') \arrow{r}{\cong}
  & \cX^R_M(P'). \\
 \end{tikzcd}
\]

We will use this observation (which is a version of the famous Yoneda Lemma in category theory) frequently throughout the paper.
A similar characterization holds, of course, for the universal pasture of $M$.

\begin{ex}\label{ex: foundations of regular matroids}
 As an example, we compute the foundation of a regular matroid. Since a regular matroid $M$ has a unique rescaling class of $P$-representations for every $P$ (which is given by a unimodular matrix), we conclude that $\Hom(F_M, P)=\cX^R_M(P)$ is a singleton for every $P$. In other words, $F_M$ has a unique morphism to any other pasture, which characterizes $F_M$ as the intial object $F_M=\Funpm$ of the category of pastures. 
 
 This holds, in particular, for the foundation $F_p=\Funpm$ of the matroid $M= U^1_1$ of rank $1$ with one element $p$.
\end{ex}

\subsubsection{Induced representations for embedded minors}

Let $\cH$ be a modular system of $P$-hyperplane functions for a matroid $M$ over a pasture $P$, and let $A \subseteq E(M)$.
For $f_H \in \cH$ and $X \subseteq E(M)$, we write $f_H|_X$ for the restriction of the function $f_H$ to $X$.
Define $\cH/A = \{f_H|_{E(M) - A} \mid A \subseteq H\}$, and define $\cH \del A = \{f_H|_{E(M) - A} \mid H - A \textrm{ is a hyperplane of } M \del A\}$.
The following was originally stated in terms of weak $P$-circuits, but we obtain the following statement via the cryptomorphism between weak $P$-circuits and $P$-hyperplane functions.

\begin{thm} \cite[Theorem 3.29]{Baker-Bowler19} \label{thm: induced representation for embedded minor}
Let $M$ be a matroid, let $P$ be  pasture, let $\cH$ be a modular system of $P$-hyperplane functions for $M$, and let $A \subseteq E(M)$.
Then, up to multiplying functions by scalars, $\cH/A$ and $\cH \del A$ are modular systems of $P$-hyperplane functions for $M/A$ and $M \del A$, respectively.
\end{thm}

An \emph{embedded minor} of a matroid $M$ is a minor $N = M \del I/J$ together with the pair $(I,J)$, where $I$ is a coindependent subset and $J$ is an independent subset of $E(M)$ such that $I \cap J = \varnothing$.
Given an embedded minor $N = M\del I/J$ and a $P$-representation of $M$ over a pasture $P$, \autoref{thm: induced representation for embedded minor} gives an induced $P$-representation for $N$.
In general, this representation depends on the choices of $I$ and $J$, meaning that if $N = M\del I/J = M \del I'/J'$, the representation induced by $(I, J)$ may not be rescaling equivalent to the representation induced by $(I', J')$.
However, when $N$ is a restriction of $M$ (or dually, a contraction of $M$), the induced representation is independent of the minor embedding.
Before proving this, we highlight the following corollary of \autoref{thm: induced representation for embedded minor}, which we will use repeatedly in our proofs.

\begin{prop} \label{prop: scaling of induced hyperplane functions}
Let $M$ be a matroid, let $T\subseteq E(M)$, let $P$ be a pasture, and let $\cH$ be a modular system of $P$-hyperplane functions for $M$.
If $H$ and $K$ are hyperplanes of $M$ so that $H \cap T = K \cap T$ and this set is a hyperplane of $M|T$, then the functions $f_H|_T$ and $f_{K}|_T$ are scalar multiples of each other.
\end{prop}

Given a matroid $M$ with $T\subseteq E(M)$, we will use \autoref{prop: scaling of induced hyperplane functions} to define an induced system of hyperplane functions for $M|T$ that is independent of the minor embedding of $M|T$.

\begin{prop} \label{prop: induced hyperplane functions}
Let $M$ be a matroid and let $T$ and $J$ be disjoint subsets of $E(M)$ so that $r(T) + r(J) = r(T \cup J)$.
Let $P$ be a pasture, and let $\cH$ be a modular system of $P$-hyperplane functions for $M$.
Let $\mathscr T_J$ be the set of hyperplanes of $M$ that contain $J$ and whose restriction to $T$ is a hyperplane of $M|T$, and let $\cH|_T = \{f_H|_T \mid H \in \mathscr T_J\}$.
Then, up to multiplying functions by scalars, $\cH|_T$ is a modular system of $P$-hyperplane functions for $M|T$, and is independent of the choice of $J$.
\end{prop}
\begin{proof}
By \autoref{prop: scaling of induced hyperplane functions} we may assume, by rescaling, that if $H$ and $K$ are hyperplanes in $\mathscr T_J$ with $H \cap T = K \cap T$, then $f_H|_T = f_K|_T$.
We will first show that every hyperplane of $M|T$ has an associated function in $\cH|_T$.
Fix a basis $B$ of $M/T$ with $J \subseteq B$.
For each hyperplane $L$ of $M|T$, the set $L' = \cl_M(L \cup B)$ is a hyperplane in $\mathscr T_J$ with $L' \cap T = L$, so $L$ has associated $P$-hyperplane function $f_{L'}|_T \in \cH|_T$.
So it suffices to show that $\cH|_T$ is a modular system.
Let $(L_1, L_2, L_3)$ be a modular triple of hyperplanes of $M|T$, and for each $i \in [3]$ let $L_i' = \cl_M(L_i \cup B)$.
Then $(L_1', L_2', L_3')$ is a modular triple of hyperplanes of $M$, so there are constants $c_1, c_2, c_3 \in P^{\times}$ so that $c_1 \cdot f_{L_1'}(e) + c_2 \cdot f_{L_2'}(e) + c_3 \cdot f_{L_3'}(e) = 0$ for all $e \in E(M)$.
Then $c_1 \cdot f_{L_1}(e) + c_2 \cdot f_{L_2}(e) + c_3 \cdot f_{L_3}(e) = 0$ for all $e \in T$, so $\cH|_T$ is a modular system of $P$-hyperplane functions for $M|T$.
Since $\{f_H|_T \mid H \in \mathscr T_J\} = \{f_H|_T \mid H \in \mathscr T_{\varnothing}\}$ because $\mathscr T_J \subseteq \mathscr T_{\varnothing}$, it follows that the modular system is independent of the choice of $J$.
\end{proof}

Given a matroid $M$ with $T \subseteq E(M)$, a pasture $P$, and a $P$-representation $\cH$ of $M$, we define $\cH|_T = \{f_H|_T \mid H \textrm{ is a hyperplane of } M|T\}$. Let $E(M)-T=I\sqcup J$ be a decomposition of the complement of $T$ in $M$ into a coindependent set $I$ and an independent set $J$. Then $M|T\simeq M\minor{I}{J}$, which induces a morphism of foundations
\[
 \iota_{M|T}: F_T \ \simeq \ F_{M\minor{I}{J}} \ \longrightarrow \ F_M
\]
where we write $F_T$ for $F_{M|T}$. 

\begin{lemma}\label{lemma: foundation map for restriction is independent on minor embedding}
 The morphism $\iota_{M|T}$ does not depend on the choices of $I$ and $J$.
\end{lemma}

\begin{proof}
 Two choices of decompositions $E(M)-T=I_i\sqcup J_i$ (for $i=1,2$) induce two morphisms $\iota_i: F_T\to F_M$, each arising from the restriction of (the rescaling classes of) a modular system of $P$-hyperplane functions of $M$ to $M|T$. Since these restrictions are independent of the choices of the decomposition $E(M)-T=I_i\sqcup J_i$, this means that the induced morphism of functors $\Hom(F_M,-)\to\Hom(F_T,-)$ is independent of $E(M)-T=I_i\sqcup J_i$. By the Yoneda lemma, this means that the morphism $F_T\to F_M$ is independent of this decomposition.
\end{proof}

As a consequence, the tensor product $F_{M_1}\otimes_{F_T}F_{M_2}$ of the foundations of two matroids $M_1$ and $M_2$ with common restriction $M_1|T=M_2|T$ has an intrinsic meaning that does not depend on the choice of minor embeddings of $M|T$ into $M_1$ and $M_2$.

\subsubsection{Cross ratios} \label{sec:cross ratios}

Let $\Omega_M$ be the collection of $5$-tuples $(J; e_1,e_2,e_3,e_4)$, where $J$ is an independent subset of $E(M)$ of cardinality $r-2$ and $e_1,e_2,e_3,e_4\in E(M)$ are elements such that $Je_ie_j$ is a basis for all $i\in\{1,2\}$ and $j\in\{3,4\}$, writing $Je_ie_j$ for $J \cup \{e_i, e_j\}$. This means in particular that $Je_i$ has rank $r-1$, and thus $H_i=\cl(Je_i)$ is a hyperplane, and that $e_j\notin H_i$ for $i\in\{1,2\}$ and $j\in\{3,4\}$.

The identification $\Hom(F_M,F_M)=\cX^R_M(F_M)$ associates with the identity map $\id:F_M\to F_M$ the 
\emph{universal rescaling class of $M$}, which is the rescaling class 
of some $F_M$-representation $\{f_H\mid H\in\cH\}$ of $M$. We define the \emph{universal cross ratio of $(J; e_1,e_2,e_3,e_4)\in\Omega_M$} as
\[
 \cross{e_1}{e_2}{e_3}{e_4}{J} \ = \ \frac{f_{H_1}(e_3)\cdot f_{H_2}(e_4)}{f_{H_1}(e_4)\cdot f_{H_2}(e_3)},
\]
where $H_i=\cl(Je_i)$. Since rescaling by $c=\big((c_e),(c_H)\big)\in (P^\times)^E\times(P^\times)^\cH$ yields 
\[
 \frac{(c f_{H_1})(e_3)\cdot (cf_{H_2})(e_4)}{(cf_{H_1})(e_4)\cdot (cf_{H_2})(e_3)} \ = \ \frac{c_{H_1}c_{e_3}f_{H_1}(e_3)\cdot c_{H_2}c_{e_4}f_{H_2}(e_4)}{c_{H_1}c_{e_4}f_{H_1}(e_4)\cdot c_{H_2}c_{e_3}f_{H_2}(e_3)} \ = \ \frac{f_{H_1}(e_3)\cdot f_{H_2}(e_4)}{f_{H_1}(e_4)\cdot f_{H_2}(e_3)},
\]
the universal cross ratio $\cross{e_1}{e_2}{e_3}{e_4}{J}$ depends only on the universal rescaling class, which shows that $\cross{e_1}{e_2}{e_3}{e_4}{J}$ is a well-defined element of $F_M$.

We have $\cross{e_1}{e_2}{e_3}{e_4}{J}=1$ if $Je_1e_2$ or $Je_3e_4$ is not a basis, i.e., if $H_1=H_2$ or $\cl(Je_3)=\cl(Je_4)$. In these cases, we say that $\cross{e_1}{e_2}{e_3}{e_4}{J}$ is \emph{degenerate}.

A more profound result, which is a consequence to Tutte's path theorem \cite[Theorem 5.1]{Tutte58a}, is that $F_M^\times$ is generated by $-1$ and all universal cross ratios \cite[Corollary 7.11]{Baker-Lorscheid21b}. Similarly, Tutte's homotopy theorem \cite[Theorem 6.1]{Tutte58a} can be used to exhibit a complete system of relations between the cross ratios as elements of $F_M^\times$ (see \cite[Theorem 4.19]{Baker-Lorscheid25}), but we won't need this latter result for our purposes.

\subsection{Facts about generalized parallel connections}

Throughout this section, let $M_1,M_2$ be matroids
with ground sets $E_1$ and $E_2$, respectively, with $E_1 \cap E_2 = T$ such that $M_1|T = M_2|T$ and $T$ is a modular flat in $M_2$.

We have the following formula for the rank of flats in $P_T(M_1, M_2)$. 

\begin{prop}\cite[Proposition 5.5]{Brylawski75}
\label{prop:rank function of the parallel connection}
If $r,r_1,r_2$ are the rank functions of $P_T(M_1,M_2)$, $M_1$, and $M_2$ respectively, then for any flat $F$ of $P_T(M_1,M_2)$ we have:
\renewcommand{\theequation}{\alph{equation}}
\setcounter{equation}{0}
\begin{equation} \label{eq:rank-of-flats-in-generalized-parallel-connection}
r(F) = r_1(F\cap E_1) +r_2(F\cap E_2) - r_1(F\cap T).
\end{equation}
In particular, 
\begin{equation} \label{eq:rank-of-generalized-parallel-connection}
r(P_T(M_1,M_2)) = r(M_1) + r(M_2) - r(M_1|T).
\end{equation}
\end{prop}

When $T$ is modular in both $M_1$ and $M_2$, there is a straightforward description of the hyperplanes of $P_T(M_1, M_2)$.

\begin{prop}\cite[Proposition 22]{HOCHSTATTLER2011841}\label{prop:hyperplanes of the parallel connection}
Assume that $T$ is a modular flat in both $M_1$ and $M_2$.
A subset $H \subseteq E_1 \cup E_2$ is a hyperplane of $P_T(M_1,M_2)$ if and only if
\begin{enumerate}
\item $H\cap E_1$ is a hyperplane of $M_1$ that contains $T$, and $H$ contains $E_2$, or
\item $H\cap E_2$ is a hyperplane of $M_2$ that contains $T$, and $H$ contains $E_1$, or
\item $H\cap E_i$ is a hyperplane of $M_i$ for $i=1,2$, and $r_{M_1}(H \cap T) = r_{M_1}(T) - 1$.
\end{enumerate}
\end{prop}
\begin{proof}
    Let $r$ be the rank function of $P_T(M_1,M_2)$.
    First suppose that $H$ is a hyperplane of $P_T(M_1, M_2)$.
    Then $H \cap E_i$ is a flat of $M_i$ for $i = 1,2$.
    Let $r(H \cap T) = r(T) - k$ where $0 \le k \le r(T)$.
    Since $T$ is a modular flat in $M_i$ we have
    \begin{align*}
        r(T) + r(H \cap E_i) &= r(T \cap H) + r((T \cup H) \cap E_i) \\
        &= r(T) - k + r((T \cup H) \cap E_i) \\
        &\le r(T) - k + r(E_i),
    \end{align*}
    and it follows that $r(H \cap E_i) \le r(E_i) - k$.
    Then we have 
    \begin{align*}
    r(H) &= r(E_1) + r(E_2) - r(T) - 1 \\
    &= r(H \cap E_1) + r(H \cap E_2) - r(H \cap T) \\
    &= r(H \cap E_1) + r(H \cap E_2) - (r(T) - k) \\
    & \le (r(E_1) - k) + (r(E_2) - k) - (r(T) - k) \\
    &= r(E_1) + r(E_2) - r(T) - k,
    \end{align*}
    where the first line follows from \eqref{eq:rank-of-generalized-parallel-connection} and the fact that $H$ is a hyperplane of $P_T(M_1, M_2)$, and the second follows from \eqref{eq:rank-of-flats-in-generalized-parallel-connection}.
    By comparing the first and last lines, we see that $k \le 1$.
    By comparing the first and third lines, we have
    \renewcommand{\theequation}{\alph{equation}}
    \begin{equation} \label{eq:rank-of-generalized-parallel-connection-hyperplane}
    r(E_1) + r(E_2) - 1 = r(H \cap E_1) + r(H \cap E_2) + k.
    \end{equation}
    If $k = 0$, then $r(H \cap T) = r(T)$, and since $H$ is a flat of $P_T(M_1, M_2)$ it follows that $T \subseteq H$.
    By (\ref{eq:rank-of-generalized-parallel-connection-hyperplane}), there is some $j \in \{1,2\}$ so that $r(H \cap E_j) = r(E_j) - 1$ and $r(H \cap E_{3-j}) = r(E_{3-j})$.
    Since $H \cap E_i$ is a flat of $M_i$ for $i = 1,2$ by the definition of $P_T(M_1, M_2)$, it follows that if $j = 1$ then (1) holds, and if $j = 2$ then (2) holds.
    If $k = 1$, then $r(H \cap T) = r(T) - 1$.
    By (\ref{eq:rank-of-generalized-parallel-connection-hyperplane}) and the observation that $r(H \cap E_i) \le r(E_i) - k$ for $i = 1,2$ we see that $r(H \cap E_i) = r(E_i) - 1$ for $i = 1,2$.
    Since $H \cap E_i$ is a flat of $M_i$ for $i = 1,2$ by the definition of $P_T(M_1, M_2)$, we see that (3) holds.

    Conversely, suppose that (1), (2), or (3) holds for $H$. 
    Since $H \cap E_i$ is a flat of $M_i$ for $i = 1,2$, it follows from the definition of $P_T(M_1, M_2)$ that $H$ is a flat of $P_T(M_1, M_2)$, so it suffices to show that $r(H) = r(P_T(M_1, M_2)) - 1$.
    If (1) or (2) holds, then by (\ref{eq:rank-of-flats-in-generalized-parallel-connection}) we see that $r(H) = r(M_1) + r(M_2) - r(T) - 1$, and it follows from (\ref{eq:rank-of-generalized-parallel-connection}) that $r(H) = r(P_T(M_1, M_2)) - 1$.
    If (3) holds, then by (\ref{eq:rank-of-flats-in-generalized-parallel-connection}) we see that $r(H) = r(M_1) + r(M_2) - r(T) - 1$, and by (\ref{eq:rank-of-generalized-parallel-connection}) it follows that $r(H) = r(P_T(M_1, M_2)) - 1$.
\end{proof}

A similar result holds for corank-2 flats.

\begin{prop} \label{prop: corank-2 flats of the parallel connection}
Assume that $T$ is a modular flat in both $M_1$ and $M_2$. A subset $F\subseteq E_1 \cup E_2$ is a corank-2 flat of $P_T(M_1,M_2)$ if and only if
\begin{enumerate}
\item $T \subseteq F$ and there is some $i \in \{1,2\}$ so that $E_i \subseteq F$ and $F \cap E_{3-i}$ is a corank-2 flat of $M_{3-i}$, or

\item $T \subseteq F$ and $F \cap E_i$ is a hyperplane of $M_i$ for $i = 1,2$, or

\item $r_{M_1}(F \cap T) = r_{M_1}(T) - 1$, and there is some $i \in \{1,2\}$ so that $F \cap E_i$ is a hyperplane of $M_i$ and $F \cap E_{3-i}$ is a corank-2 flat of $M_{3-i}$, or

\item $r_{M_1}(F \cap T) = r_{M_1}(T) - 2$, and $F \cap E_i$ is a corank-2 flat of $M_i$ for $i = 1,2$.
\end{enumerate}
\end{prop}
\begin{proof}
    Let $r$ be the rank function of $P_T(M_1,M_2)$.
    First suppose that $F$ is a corank-2 flat of $P_T(M_1, M_2)$. 
    Then $F \cap E_i$ is a flat of $M_i$ for $i = 1,2$.
    Let $r(F \cap T) = r(T) - k$ where $0 \le k \le r(T)$.
    As in the proof of \autoref{prop:hyperplanes of the parallel connection}, 
    we know that $r(F \cap E_i) \le r(E_i) - k$ for $i = 1,2$.
    Then we have 
    \begin{align*}
    r(F) &= r(E_1) + r(E_2) - r(T) - 2 \\
    &= r(F \cap E_1) + r(F \cap E_2) - r(F \cap T) \\
    &= r(F \cap E_1) + r(F \cap E_2) - (r(T) - k) \\
    & \le (r(E_1) - k) + (r(E_2) - k) - (r(T) - k) \\
    &= r(E_1) + r(E_2) - r(T) - k,
    \end{align*}
    where the first line follows from \eqref{eq:rank-of-generalized-parallel-connection} and the fact that $F$ is a corank-2 flat of $P_T(M_1, M_2)$, 
    and the second follows from \eqref{eq:rank-of-flats-in-generalized-parallel-connection}.
    By comparing the first and last lines, we see that $k \le 2$.
    By comparing the first and third lines, we have
    \renewcommand{\theequation}{\alph{equation}}
    \begin{equation} \label{eq:rank-of-generalized-parallel-connection-flat}
    r(E_1) + r(E_2) - 2 = r(F \cap E_1) + r(F \cap E_2) + k.
    \end{equation}
     If $k = 0$, then $T \subseteq F$.
    By (\ref{eq:rank-of-generalized-parallel-connection-flat}), either there is some $j \in \{1,2\}$ so that $r(F \cap E_j) = r(E_j) - 1$ and $r(F \cap E_{3-j}) = r(E_{3-j})$ and (1) holds because $F \cap E_i$ is a flat for $i = 1,2$, or $r(F \cap E_j) = r(E_j) - 1$ for $j = 1,2$ and (2) holds.
    If $k = 1$, then $r(F \cap T) = r(T) - 1$.
    By (\ref{eq:rank-of-generalized-parallel-connection-flat}) and the observation that $r(F \cap E_i) \le r(E_i) - k$ for $i = 1,2$ we see that (3) holds.
    If $k = 2$, then $r(F \cap T) = r(T) - 2$.
    By (\ref{eq:rank-of-generalized-parallel-connection-flat}) and the observation that $r(F \cap E_i) \le r(E_i) - k$ for $i = 1,2$ we see that (4) holds.

    Conversely, suppose that (1), (2), (3), or (4) holds for $F$.
    Since $F \cap E_i$ is a flat of $M_i$ for $i = 1,2$, it follows from the definition of $P_T(M_1, M_2)$ that $F$ is a flat of $P_T(M_1, M_2)$.
    In each case it follows directly from (\ref{eq:rank-of-flats-in-generalized-parallel-connection}) that $r(F) = r(M_1) + r(M_2) - r(T) - 2$, and by (\ref{eq:rank-of-generalized-parallel-connection}) it follows that $F$ is a corank-2 flat of $P_T(M_1, M_2)$.
\end{proof}

We will also need analogous results when $r(T) = 2$ and $T$ is not assumed to be modular in $M_1$.
We replace $T$ with $X$ here, because we will apply this result in the case that $M_2 = \Theta_n$.

\begin{prop} \label{Prop:hyperplanes of the parallel connection (r(T) = 2)} 
Let $M_1,M_2$ be matroids with ground sets $E_1$ and $E_2$, respectively, with $E_1 \cap E_2 = X$ such that $M_1|X = M_2|X$ and $X$ is a modular flat in $M_2$.
Assume furthermore that $M_2|X \cong U_{2,n}$ for some $n \ge 2$. 
A subset $H \subseteq E_1 \cup E_2$ is a hyperplane of $P_X(M_1, M_2)$ if and only if  
     \begin{enumerate}
         \item $E_1 \subseteq H$ and $H \cap E_2$ is a hyperplane of $M_2$ that contains $X$, or
         
         \item $E_2 \subseteq H$ and $H \cap E_1$ is a hyperplane of $M_1$ that contains $X$, or

         \item $H \cap E_i$ is a hyperplane of $M_i$ for $i = 1,2$ and $|H \cap X| = 1$, or
         
         \item $H \cap E_1$ is a hyperplane of $M_1$ that is disjoint from $X$, and $H \cap E_2$ is a corank-2 flat of $M_2$ that is disjoint from $X$.
     \end{enumerate}
\end{prop}
\begin{proof}
    Let $r$ be the rank function of $P_X(M_1,M_2)$.
    First suppose that $H$ is a hyperplane of $P_X(M_1, M_2)$.
    Then $H \cap E_i$ is a flat of $M_i$ for $i = 1,2$.
    Let $r(H \cap X) = r(X) - k$ where $0 \le k \le 2$.
    Then we have 
    \begin{align*}
    r(H) &= r(E_1) + r(E_2) - r(X) - 1 \\
    &= r(H \cap E_1) + r(H \cap E_2) - r(H \cap X) \\
    &= r(H \cap E_1) + r(H \cap E_2) - (r(X) - k),
    \end{align*}
    where the first line follows from \eqref{eq:rank-of-generalized-parallel-connection} and the fact that $H$ is a hyperplane of $P_X(M_1, M_2)$, and the second follows from \eqref{eq:rank-of-flats-in-generalized-parallel-connection}.
    It follows that 
    \renewcommand{\theequation}{\alph{equation}}
    \begin{equation} \label{eq:rank-of-generalized-parallel-connection-X}
    r(E_1) + r(E_2) - 1 = r(H \cap E_1) + r(H \cap E_2) + k.
    \end{equation}
    If $k = 0$ then $X \subseteq H$ because $H$ is a flat, and it follows from (\ref{eq:rank-of-generalized-parallel-connection-X}) that (1) or (2) holds.
    If $k = 1$ then $|H \cap X| = 1$ because $M_i|X$ is simple, and $r(H \cap E_i) < r(E_i)$ for $i = 1,2$ because $H$ does not contain $X$.
    Then it follows from (\ref{eq:rank-of-generalized-parallel-connection-X}) that (3) holds.
    Finally, if $k = 2$ then $X \cap H  = \varnothing$.
    Since $X$ is modular flat in $M_2$, we know that 
    $$r(X) + r(H \cap E_2) = r(X \cup (H \cap E_2)),$$
    and since $r(X) = 2$ and $r(X \cup (H \cap E_2)) \le r(E_2)$
    it follows that $r(H \cap E_2) \le r(E_2) - 2$.
    Since $X \cap H = \varnothing$ we know that $r(H \cap E_1) \le r(E_1) - 1$, and now (\ref{eq:rank-of-generalized-parallel-connection-X}) implies that (4) holds.

    Conversely, suppose that (1), (2), (3), or (4) holds for $H$.
    Since $H \cap E_i$ is a flat of $M_i$ for $i = 1,2$, it follows from the definition of $P_T(M_1, M_2)$ that $H$ is a flat of $P_T(M_1, M_2)$, so it suffices to show that $r(H) = r(P_T(M_1, M_2)) - 1$.
    In each case it follows directly from (\ref{eq:rank-of-flats-in-generalized-parallel-connection}) that $r(H) = r(M_1) + r(M_2) - r(T) - 1$, and then (\ref{eq:rank-of-generalized-parallel-connection}) implies that $r(H) = r(P_T(M_1, M_2)) - 1$.
\end{proof}

A similar result holds for corank-2 flats.

\begin{prop} \label{Prop:corank-2 flats of the parallel connection (r(T) = 2)}
With hypotheses as in \autoref{Prop:hyperplanes of the parallel connection (r(T) = 2)},
a subset $F \subseteq E_1 \cup E_2$ is a corank-2 flat of $P_X(M_1, M_2)$ if and only if  
     \begin{enumerate}
         \item $E_1 \subseteq F$ and $F \cap E_2$ is a corank-2 flat of $M_2$ that contains $X$, 
         
         \item $E_2 \subseteq F$ and $F \cap E_1$ is a corank-2 flat of $M_1$ that contains $X$, 

         \item For each $i = 1,2$, $F \cap E_i$ is a hyperplane of $M_i$ that contains $X$, 

         \item $|F \cap X| = 1$, $F \cap E_1$ is a hyperplane of $M_1$, and $F \cap E_2$ is a corank-2 flat of $M_2$, or 

         \item $|F \cap X| = 1$, $F \cap E_1$ is a corank-2 flat of $M_1$, and $F \cap E_2$ is a hyperplane of $M_2$,
         
         \item $F \cap X = \varnothing$, $F \cap E_1$ is a hyperplane of $M_1$, and $F \cap E_2$ is a corank-3 flat of $M_2$, or

         \item $F \cap X = \varnothing$, and $F \cap E_i$ is a corank-2 flat of $M_i$ for $i = 1,2$.
     \end{enumerate}
\end{prop}
\begin{proof}
    Let $r$ be the rank function of $P_X(M_1,M_2)$.
    First suppose that $F$ is a corank-2 flat of $P_T(M_1, M_2)$.
    Then $F \cap E_i$ is a flat of $M_i$ for $i = 1,2$.
    Let $r(F \cap X) = r(X) - k$ where $0 \le k \le 2$.
    Then we have 
    \begin{align*}
    r(F) &= r(E_1) + r(E_2) - r(X) - 2 \\
    &= r(F \cap E_1) + r(F \cap E_2) - r(F \cap X) \\
    &= r(F \cap E_1) + r(F \cap E_2) - (r(X) - k),
    \end{align*}
       where the first line follows from \eqref{eq:rank-of-generalized-parallel-connection} and the fact that $F$ is a corank-2 flat of $P_X(M_1, M_2)$, 
    and the second follows from \eqref{eq:rank-of-flats-in-generalized-parallel-connection}.
    It follows that 
    \renewcommand{\theequation}{\alph{equation}}
    \begin{equation} \label{eq:rank-of-generalized-parallel-connection-X-F}
    r(E_1) + r(E_2) - 2 = r(F \cap E_1) + r(F \cap E_2) + k.
    \end{equation}
    If $k = 0$ then $X \subseteq F$ because $F$ is a flat, and (\ref{eq:rank-of-generalized-parallel-connection-X-F}) implies that (1), (2), or (3) holds.
    If $k = 1$ then $|F \cap X| = 1$ because $M_i|X$ is simple, and $r(F \cap E_i) < r(M_i)$ for $i = 1,2$ because $F$ does not contain $X$.
    Then (\ref{eq:rank-of-generalized-parallel-connection-X-F}) implies that (4) or (5) holds.
    Finally, if $k = 2$ then $X \cap F  = \varnothing$.
    Since $X$ is modular flat in $M_2$, we know that 
    $$r(X) + r(F \cap E_2) = r(X \cup (F \cap E_2)),$$
    and since $r(X) = 2$ and $r(X \cup (F \cap E_2)) \le r(E_2)$
    it follows that $r(F \cap E_2) \le r(E_2) - 2$.
    Since $X \cap F = \varnothing$ we know that $r(F \cap E_1) \le r(E_1) - 1$, and now (\ref{eq:rank-of-generalized-parallel-connection-X}) implies that (6) or (7) holds.
    
    Conversely, suppose that one of (1)--(7) holds for $F$.
    Since $F \cap E_i$ is a flat of $M_i$ for $i = 1,2$, it follows from the definition of $P_T(M_1, M_2)$ that $F$ is a flat of $P_T(M_1, M_2)$, so it suffices to show that $r(F) = r(P_T(M_1, M_2)) - 2$.
    In each case it follows directly from (\ref{eq:rank-of-flats-in-generalized-parallel-connection}) that $r(F) = r(M_1) + r(M_2) - r(T) - 2$, and then (\ref{eq:rank-of-generalized-parallel-connection}) implies that $r(F) = r(P_T(M_1, M_2)) - 2$.
\end{proof}

We will also need to understand interactions between hyperplanes of a matroid.
Given a matroid $M$, a \emph{linear subclass} is a set $\mathscr H$ of hyperplanes of $M$ so that if $H, H' \in \mathscr H$ and $(H, H')$ is a modular pair, then every hyperplane containing $H \cap H'$ is also in $\mathscr H$.
The canonical example of a linear subclass is the set of hyperplanes containing a fixed flat.
The following proposition will be useful for inductive arguments involving hyperplanes that avoid a fixed linear subclass.

\begin{prop} \label{prop: linear subclasses}
Let $M$ be a matroid and let $\mathscr H$ be a linear subclass of $M$.
If $H$ and $K$ are distinct hyperplanes of $M$ with $H, K \notin \mathscr H$, then there is a hyperplane $L$ of $M$ so that $L \notin \mathscr H$, the pair $(H, L)$ is modular, $(H \cap K) \subseteq L$, and $r(L \cap K) > r(H \cap K)$.
\end{prop}
\begin{proof}
Let $F$ be a corank-2 flat of $M$ with $(H \cap K) \subseteq F \subseteq H$.
Let $\mathscr F$ be the set of hyperplanes of $M$ that contain $F$ and some element of $K - F$. 
If $|\mathscr F| = 1$, then $\cl_M(F \cup e) = \cl_M(F \cup e')$ for all $e,e' \in K - F$.
Then $K \subseteq \cl_M(F \cup e)$, so $(H, K)$ is a modular pair and the claim holds with $L = K$.
So we may assume that $|\mathscr F| \ge 2$.
If $\mathscr F \subseteq \mathscr H$, then $H \in \mathscr H$ because $\mathscr H$ is a linear subclass and all of the hyperplanes in $\mathscr F$ contain the corank-2 flat $F$.
So there is some $L \in \mathscr F - \mathscr H$.
Since $(H \cap K) \subseteq F \subseteq L$ we see that $(H, L)$ is a modular pair and $(H \cap K) \subseteq L$.
Since $L$ contains an element in $K - F$ we see that $r(L \cap K) > r(H \cap K)$, and the statement holds.
\end{proof}

\section{The foundation of a generalized parallel connection}
\label{sec:GPC}


The following theorem implies \autoref{thmA}~(1), and also proves the analogous result for universal pastures. Recall from \autoref{lemma: foundation map for restriction is independent on minor embedding} that the restriction of a matroid $M$ to a subset $T$ induces a (well-defined) morphism $F_{M|T}\to F_M$ of foundations.
We will write $\tilde F_T$ and $F_T$ for $\tilde F_{M|T}$ and $F_{M|T}$, respectively.

\begin{thm}\label{thm: foundation of the generalized parallel connection}
Let $M_1$ and $M_2$ be matroids with ground sets $E_1$ and $E_2$, respectively, with $E_1 \cap E_2 = T$ so that $M_1|T = M_2|T$ and $T$ is a modular flat of both $M_1$ and $M_2$, and let $M = P_T(M_1, M_2)$.
Then $\tilde F_M \cong \tilde{F}_{M_1}\otimes_{\tilde{F}_T}\tilde{F}_{M_2}$ and $F_M \cong F_{M_1}\otimes_{F_T}F_{M_2}$.
\end{thm}
\begin{proof}
Let $P$ be a pasture. 
Let $\cX^I(M_1, M_2, T, P)$ (resp. $\cX^R(M_1, M_2, T, P)$) be the subset of $\cX^I_{M_1}(P) \times \cX^I_{M_2}(P)$ (resp. $\cX^R_{M_1}(P) \times \cX^R_{M_2}(P)$)
for which the induced representations of $M_1|T$ and $M_2|T$ are in the same isomorphism class (resp. rescaling equivalence class).
We will define a map $\Phi$ from $\cX^I(M_1, M_2, T, P)$ to $\cX^I_{M}(P)$ and a map $\Psi$ from $\cX^I_{M}(P)$ to $\cX^I(M_1, M_2, T, P)$. 
Then we will show that these maps are well-defined and inverse to each other.
It will be clear from the definition of the resulting bijection that it is functorial in $P$. Therefore, by the universal property of the tensor product, we will obtain an isomorphism $\tilde F_M \cong \tilde{F}_{M_1}\otimes_{\tilde{F}_T}\tilde{F}_{M_2}$. 
Passing to rescaling classes instead of isomorphism classes shows that $F_M \cong F_{M_1}\otimes_{F_T}F_{M_2}$ as well.

For $\Phi$, if we have a modular system $\cH$ of $P$-hyperplane functions of $M$, then $\cH|_{E_1}$ and $\cH|_{E_2}$ are modular systems of hyperplane functions for $M_1$ and $M_2$, respectively, whose induced representations of $M|T$ are clearly isomorphic.
For $\Psi$, let $\cH_i$ be a modular system of $P$-hyperplane functions of $M_i$ for $i =1,2$ so that $\cH_1|_T$ and $\cH_2|_T$ are isomorphic.
By Propositions \ref{prop: scaling of induced hyperplane functions} and \ref{prop:hyperplanes of the parallel connection}, we may assume, by scaling functions in $\cH_1$ and $\cH_2$, that if $f, f' \in \cH_1 \cup \cH_2$ have the same support in $T$, then $f(e) = f'(e)$ for all $e \in T$.
For each hyperplane $H$ of $M$ we define a function $f_H$ by declaring that if $H \cap E_i$ is a hyperplane for some $i=1,2$, then $f_H(e) = f_{H \cap E_i}(e)$ for all $e \in E_i$.
Let $\cH$ be the set of all $f_H$ for hyperplanes $H$ of $M$.
By \autoref{prop:hyperplanes of the parallel connection}, the complements of the supports of the functions in $\cH$ forms the set of hyperplanes of $M$. 
Clearly $\Phi$ and $\Psi$ are inverse to each other because restricting the functions in $\cH$ to $E_i$ for $i = 1,2$ results in the systems $\cH_1$ and $\cH_2$.
So it remains to show that $\cH$ is in fact a modular system. 

Let $F$ be a corank-2 flat of $M$ and let $(H, H', H'')$ be a modular triple of hyperplanes of $M$ such that $H \cap H' \cap H'' = F$.
We will show that $f_{H}, f_{H'}, f_{H''}$ are linearly dependent.
There are four different cases to consider, stemming from the four cases for $F$ in \autoref{prop: corank-2 flats of the parallel connection}.

\textbf{Case 1:} Suppose $T \subseteq F$ and there is some $i \in \{1,2\}$ so that $E_i \subseteq F$ and $F \cap E_{3-i}$ is a corank-2 flat of $M_{3-i}$.
We may assume that $i = 1$.
Then $(H \cap E_2, H' \cap E_2, H'' \cap E_2)$ is a modular triple of hyperplanes of $M_2$, and since $f_{H \cap E_2}, f_{H' \cap E_2}, f_{H'', \cap E_2}$ are linearly dependent in $\cH_2$ it follows that $f_{H}, f_{H'}, f_{H''}$ are linearly dependent in $\cH$.

\textbf{Case 2:} Suppose $T \subseteq F$ and $F \cap E_i$ is a hyperplane of $M_i$ for $i = 1,2$.
By \autoref{prop:hyperplanes of the parallel connection}, the only hyperplanes of $M$ containing $F$ are $F \cup E_1$ and $F \cup E_2$, so there is no modular triple of hyperplanes that all contain $F$.

\textbf{Case 3:} Suppose $r_{M_1}(F \cap T) = r_{M_1}(T) - 1$, and there is some $i \in \{1,2\}$ so that $F \cap E_i$ is a hyperplane of $M_i$ and $F \cap E_{3-i}$ is a corank-2 flat of $M_{3-i}$.
We may assume that $i = 1$.
By \autoref{prop:hyperplanes of the parallel connection} we see that $(H \cap E_2, H' \cap E_2, H'' \cap E_2)$ is a modular triple of hyperplanes of $M_2$, so there are constants $c,c',c''$ so that 
$$c \cdot f_{H \cap E_2}(e) + c'\cdot f_{H' \cap E_2}(e) + c''\cdot f_{H''\cap E_2}(e) = 0$$ for all $e \in E_2$.
If none of $H, H', H''$ contains $E_1$, then $c + c' + c'' = 0$ because $H \cap E_2, H' \cap E_2, H'' \cap E_2$ all have the same restriction to $T$.
Similarly, if $E_1 \subseteq H''$, then $c + c' = 0$.
In either case it follows that $c \cdot f_{H}(e) + c'\cdot f_{H'}(e) + c''\cdot f_{H''}(e) = 0$ for all $e \in E_1 \cup E_2$.

\textbf{Case 4:} Suppose $r_{M_1}(F \cap T) = r_{M_1}(T) - 2$ and $F \cap E_i$ is a corank-2 flat of $M_i$ for $i = 1,2$.
If outcome (1) or (2) of \autoref{prop:hyperplanes of the parallel connection} holds for $H$, then by (\ref{eq:rank-of-flats-in-generalized-parallel-connection}) we see that $r_M(H) \ge r_M(F) + 2$, a contradiction.
So outcome (3) of \autoref{prop:hyperplanes of the parallel connection} holds for $H$, $H'$, and $H''$, and since $F \cap E_i$ is a corank-2 flat of $M_i$ for $i = 1,2$, it follows that $(H \cap E_i, H' \cap E_i, H'' \cap E_i)$ is a modular triple of hyperplanes of $M_i$ for $i = 1,2$.
Then there are constants $c, c, c''$ so that 
$$c \cdot f_{H \cap E_1}(e) + c'\cdot f_{H' \cap E_1}(e) + c''\cdot f_{H''\cap E_1}(e) = 0$$ for all $e \in E_1$,
and constants $d, d', d''$ so that 
$$d \cdot f_{H \cap E_2}(e) + d'\cdot f_{H' \cap E_2}(e) + d''\cdot f_{H''\cap E_2}(e) = 0$$ for all $e \in E_2$.
Since outcome (3) of \autoref{prop:hyperplanes of the parallel connection} holds for $H$, $H'$, and $H''$, we know that $r_M(H \cap T) = r_M(H' \cap T) = r_M(H'' \cap T) = r_M(T) - 1$.
Since $F$ and $H$ do not agree on $T$, there is an element $t \in (H \cap T) - F$ so that $\cl_M(F \cup t) = H$.
Then $t \notin H' \cup H''$, or else $H = H'$ or $H = H''$.
By setting $e = t$, the first equation shows that $\frac{c}{c'} = -\frac{f_{H' \cap E_1}(t)}{f_{H \cap E_1}(t)}$, and the second equation shows that $\frac{d}{d'} = -\frac{f_{H' \cap E_2}(t)}{f_{H \cap E_2}(t)}$.
It follows that $\frac{c}{c'} = \frac{d}{d'}$.
Repeating this argument with an element $t' \in (H' \cap T) - (H \cup H'')$ and an element $t'' \in (H'' \cap T) - (H \cup H')$ shows that $(c, c', c'')$ is a scalar multiple of $(d, d', d'')$, and it follows that 
$c \cdot f_{H}(e) + c'\cdot f_{H'}(e) + c''\cdot f_{H''}(e) = 0$ for all $e \in E_1 \cup E_2$.

The four cases combine to show that $\cH$ is a modular system of $P$-hyperplane functions for $M$, as desired.
So we have defined maps from $\cX^I(M_1, M_2, T, P)$ to $\cX^I_{M}(P)$ and vice versa that are inverse to each other and functorial in $P$, which shows that $\tilde F_M \cong \tilde{F}_{M_1}\otimes_{\tilde{F}_T}\tilde{F}_{M_2}$.
Since these maps induce maps from $\cX^R(M_1, M_2, T, P)$ to $\cX^R_{M}(P)$ and vice versa that are also inverse to each other and functorial in $P$, it follows that $F_M \cong F_{M_1}\otimes_{F_T}F_{M_2}$ as well.
\end{proof}

\begin{rem}
When $T$ is only a modular flat in $M_2$, the generalized parallel connection $P_T(M_1, M_2)$ is still well-defined.
However, the identity $F_{P_T(M_1, M_2)}\cong F_{M_1} \otimes_{F_T} F_{M_2}$ does not always hold in this more general setting, even when $r(T) = 2$.
For example, let $M_1$ and $M_2$ be the rank-$3$ matroids spanned by the two planes of the matroid shown in \autoref{fig: tensor product counterexample}, and let $T$ be the line spanned by the intersection of these two planes.
Then $T$ is a modular flat of $M_2$, so $M = P_T(M_1, M_2)$ is well-defined.
However, one can check,
using the Macaulay2 package developed by Chen and Zhang (cf.~\cite{Chen-Zhang})\footnote{The software described in \cite{Chen-Zhang} is now available through the standard distribution of Macaulay2 as the package ``foundations.m2''.},
that $F_M \ncong F_{M_1} \otimes_{F_T} F_{M_2}$.
Specifically, $F_{M_1} \otimes_{F_{T}} F_{M_2}$ has $30$ hexagons (in the sense of \cite[Figure 4.1]{Baker-Lorscheid25}) while $F_M$ has  $31$ hexagons.

We briefly explain how this extra hexagon in $F_M$ arises from the fact that $T$ is not a modular flat of $M_1$.
Let $H = E(M_1) - T$ and let $\{a,b,c,d\} = E(M_2) - T$.
Then $H \cup a$, $H \cup b$, $H \cup c$, and $H \cup d$ are all hyperplanes of $M$ that are not of the form described in \autoref{prop:hyperplanes of the parallel connection}.
Moreover, $(H \cup a, H \cup b, H \cup c, H \cup d)$ is a modular quadruple of hyperplanes of $M$, which corresponds to a hexagon of $F_M$ (see \cite[Definitions 3.3 and 3.4]{Baker-Lorscheid25}).
The pasture obtained from $F_M$ by deleting this hexagon is isomorphic to $F_{M_1} \otimes_{F_T} F_{M_2}$ (as verified via Macaulay2), so the discrepancy between $F_M$ and $F_{M_1} \otimes_{F_T} F_{M_2}$ arises directly from the fact that $T$ is not a modular flat of $M_1$.
\end{rem}

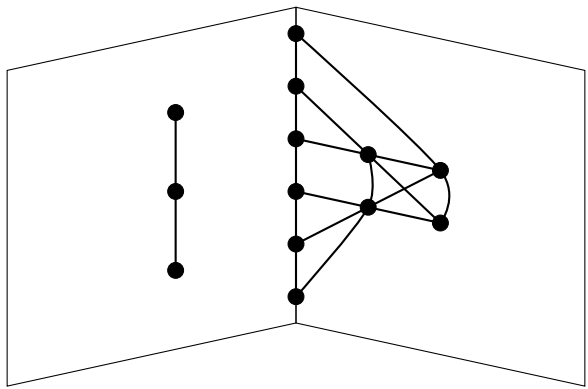
\begin{figure}[htb] 
\[
 \beginpgfgraphicnamed{tikz/fig49}
 \begin{tikzpicture}[x=0.08cm,y=0.07cm]
  \draw ( 0,-5) -- ( 0,55) -- (48,43) -- (48,-17) --cycle;
  \draw ( 0,-5) -- ( 0,55) -- (-48,43) -- (-48,-17) --cycle;
  \draw[thick] ( 0, 0) -- ( 0,50);
  \draw[thick] ( 0, 20) -- ( 24,14);
  \draw[thick] ( 0, 30) -- ( 24,24);
  \draw[thick] ( -20, 5) -- ( -20,35);
  \draw[thick] ( 24, 14) -- ( 0,40);
  \draw[thick] ( 24, 24) -- ( 0,10);
  \draw[thick] plot [smooth] coordinates { (24,14) ( 24,24) (0,50) };
  \draw[thick] plot [smooth] coordinates { (12,27) ( 12,17) (0,0) };
  \filldraw ( 0, 0) circle (3pt);
  \filldraw ( 0,10) circle (3pt);
  \filldraw ( 0,20) circle (3pt);
  \filldraw ( 0,30) circle (3pt);
  \filldraw ( 0,40) circle (3pt);
  \filldraw ( 0,50) circle (3pt);
  \filldraw (12,27) circle (3pt);
  \filldraw (24,24) circle (3pt);
  \filldraw (12,17) circle (3pt);
  \filldraw (24,14) circle (3pt);
  \filldraw (-20, 5) circle (3pt);
  \filldraw (-20,20) circle (3pt);
  \filldraw (-20,35) circle (3pt);
 \end{tikzpicture}
 \endpgfgraphicnamed
 \]
 \caption{A generalized parallel connection for which the foundation of $P_T(M_1, M_2)$ is not isomorphic to $F_{M_1} \otimes_{F_T} F_{M_2}$.} 
 \label{fig: tensor product counterexample}
\end{figure}

\section{The foundation of a 2-sum} \label{sec:foundation2sum}

In this section, we study the special case in which $T = \{p\}$ is a singleton that is not a loop or a coloop in either $M_1$ or $M_2$. 
In this case, the \emph{2-sum} of $M_1$ and $M_2$ with basepoint $p$ is the matroid with ground set $(E(M_1) \cup E(M_2)) - p$ and set of circuits
\[
\cC(M_1\setminus p) \cup \cC(M_2\setminus p) \cup \{(C_1\cup C_2) -  p \mid p\in C_1 \in \cC(M_1) \text{ and } p \in C_2\in \cC(M_2)\},
\]
where $\cC(N)$ denotes the set of circuits of the matroid $N$.
The 2-sum of $M_1$ and $M_2$ with basepoint $p$ is denoted by $M_1\oplus_2 M_2$ or $M_1\oplus_p M_2$.
When $\{p\}$ is a flat of $M_1$ or $M_2$, we can also define $M_1\oplus_p M_2$ to be $P_{p}(M_1, M_2) \del p$, where $P_{p}(M_1, M_2)$ is the parallel connection of $M_1$ and $M_2$ along $p$ \cite[Proposition 7.1.20]{Oxley06}.

We seek to prove \autoref{thmB}, which states that $F_{M_1 \oplus_p M_2}\cong F_{M_1}\otimes F_{M_2}$, where we use that $F_p=\Funpm$, as noted in \autoref{ex: foundations of regular matroids}.
We know from \autoref{thm: foundation of the generalized parallel connection} that $F_{P_p(M_1, M_2)} \cong F_{M_1}\otimes F_{M_2}$, so it suffices to show that $F_{M_1 \oplus_p M_2}\cong F_{P_p(M_1, M_2)}$.
We first show that the sets of hyperplanes of $M_1 \oplus_p M_2$ and $P_p(M_1, M_2)$ are closely related.

\begin{lemma} \label{lem: hyperplanes of M'}
Let $M_1$ and $M_2$ be matroids on $E_1$ and $E_2$, respectively, so that $E_1 \cap E_2 = \{p\}$ where $p$ is not a loop or a coloop of $M_1$ or $M_2$, and $\{p\}$ is a flat of $M_1$ or $M_2$.
Let $M = P_{p}(M_1, M_2)$ and $M' = M_1 \oplus_p M_2$, and let $\mathscr H$ and $\mathscr H'$ be the sets of hyperplanes of $M$
and $M'$ respectively. 
Then
\begin{enumerate}
    \item $\mathscr H' = \{H-p\mid H\in \cH\}$, 

    \item if $(H_1, H_2, H_3)$ is a modular triple of hyperplanes of $M$, then $(H_1 - p, H_2 - p, H_3 - p)$ is a modular triple of hyperplanes of $M'$, and

    \item conversely, if $(H_1', H_2', H_3')$ is a modular triple of hyperplanes of $M'$, then $$(\cl_M(H_1'), \cl_M(H_2'), \cl_M(H_3'))$$ is a modular triple of hyperplanes of $M$.
\end{enumerate}
\end{lemma}
\begin{proof}
We first prove (1).
Since $M$ is obtained from $M'$ by deleting $p$, it follows that $\mathscr H' \subseteq \{H-p\mid H\in \mathscr H\}$, so we need only show that the reverse containment holds as well. 
If $p\notin H$ then clearly $H-p\in \mathscr H'$. If $p\in H$ then $E_i \subseteq H$ for some $i\in \{1,2\}$ by \autoref{prop:hyperplanes of the parallel connection}. Since $p$ is not a coloop of $M_i$, it follows that $H$ and $H-p$ have the same rank in $M$, and so $H-p\in \mathscr H'$.

We next prove (2).
Suppose $(H_1,H_2, H_3)$ is a modular triple of hyperplanes of $M$.
Let $L = H_1 \cap H_2 \cap H_3$.
It suffices to show that if $p \in L$, then $r_M(L - p) = r_M(L)$.
If $p\in L$, then by \autoref{prop:hyperplanes of the parallel connection}, each of $H_1$, $H_2$, and $H_3$ contains $E_1$ or $E_2$. 
If $E_1 \in H_1$ and $E_2 \in H_2$ then $(H_1, H_2, H_3)$ is not a modular triple, so without loss of generality we may assume $E_1\subseteq L$.
Since $p$ is not a coloop of $M_1$, it follows that $r_M(L - p) = r_M(L)$, as desired.

Finally, we prove (3).
Suppose $(H_1', H_2', H_3')$ is a modular triple of hyperplanes of $M'$, and let $L' = H_1' \cap H_2' \cap H_3'$.
Then $$r(M) - 2 = r(M') - 2 = r_{M'}(L') = r_M(\cl_M(L')) = r_M(\cl_M(H_1') \cap \cl_M(H_2') \cap \cl_M(H_3')),$$
which shows that $(\cl_M(H_1'), \cl_M(H_2'), \cl_M(H_3'))$ is a modular triple of hyperplanes of $M$.
\end{proof}

The following is a restatement of \autoref{thmB}.

\begin{thm}\label{thm: foundations of 2-sums}
Let $M_1$ and $M_2$ be matroids on $E_1$ and $E_2$, respectively, so that $E_1 \cap E_2 = \{p\}$ and $p$ is not a loop or a coloop of $M_1$ or $M_2$.
Then $F_{M_1 \oplus_p M_2} \cong F_{M_1}\otimes F_{M_2}$.
\end{thm}
\begin{proof}
We will write $\si(M)$ for the simplification of a matroid $M$.
We will reduce to the case in which $M_1$ and $M_2$ are simple.
First suppose that $p$ is a coloop of $\si(M_i)$ for some $i \in \{1,2\}$; we may assume that $i = 1$.
Since $p$ is not a coloop of $M_1$, it is in a nontrivial parallel class of $M_1$.
By \cite[Proposition 7.1.15~(v)]{Oxley06} and \cite[Corollary 4.10]{Baker-Lorscheid25} we may assume that this parallel class is $\{p, p'\}$ for some $p' \in E_1$.
Then $p'$ is a coloop of $M \del p$, so $\cC(M_1 \del p) = \cC(M_1\del \{p, p'\})$, and $\{p, p'\}$ is the unique circuit of $M_1$ that contains $p$.
Let $\hat M_2$ be the matroid obtained from $M_2$ by adding $p'$ in parallel to $p$ and then deleting $p$.
Then $E(\hat M_2) = (E_2 - p) \cup p'$, and clearly $\hat M_2 \cong M_2$ and $\cC(M_2 \del p) = \cC(\hat M_2 \del p')$.
Since $\{p, p'\}$ is the unique circuit of $M_1$ that contains $p$, we see that 
\begin{align*}
    \{(C_1\cup C_2) -  p \mid p\in C_1 \in \cC(M_1) \text{ and } p \in C_2\in \cC(M_2)\}
\end{align*}
is equal to $\{C_2 \in \cC(\hat M_2) \mid p' \in C_2\}$, because the only choice for $C_1$ is $\{p, p'\}$.
Since $\cC(M_1 \del p) = \cC(M_1\del \{p, p'\})$ and $\cC(M_2 \del p) = \cC(\hat M_2 \del p)$, it follows that 
\begin{align*}
    \cC(M_1 \oplus_p M_2) &= \cC(M_1\del \{p, p'\}) \cup \cC(\hat M_2 \del p) \cup \{C_2 \in \cC(\hat M_2) \mid p' \in C_2\} \\
    &= \cC(M_1\del \{p, p'\}) \cup \cC(\hat M_2) \\
    & = \cC(M_1 \del \{p, p'\} \oplus \hat M_2),
\end{align*}
where the last equality is due to \cite[4.2.12]{Oxley06}.
Therefore $M_1 \oplus_p M_2 \cong M_1 \del \{p, p'\} \oplus \hat M_2$.
Since $F_{M_1 \del \{p, p'\}} \cong F_{M_1}$ by \cite[Corollary 4.10]{Baker-Lorscheid25} and $F_{\hat M_2} \cong F_{M_2}$ because $\hat M_2 \cong M_2$, it follows from Corollary \ref{corA} that $F_{M_1 \oplus_p M_2} \cong F_{M_1} \otimes F_{M_2}$.
So we may assume that $p$ is not a coloop of $\si(M_1)$ or $\si(M_2)$.
Then it follows from \cite[Proposition 7.1.15~(v)]{Oxley06} that $\si(M_1 \oplus_p M_1) = \si(M_1) \oplus_p \si(M_2)$, so by \cite[Corollary 4.10]{Baker-Lorscheid25} we may assume that $M_1$ and $M_2$ are simple.

Let $E = E_1 \cup E_2$, let $E' = E - p$, and let $E_i' = E_i - p$ for $i = 1,2$.
Let $P$ be a pasture.
Given functions $f_i \colon E_i' \to P$ for $i = 1,2$, we define $f_1 * f_2$ to be the function from $E'$ to $P$ so that $(f_1 * f_2)(e) = f_i(e)$ when $e \in E_i'$.
Using modular systems of hyperplane functions, we first define a map $\Phi$ from $\cX^R_{M}(P)$ to $\cX^R_{M'}(P)$ and a map $\Psi$ from $\cX^R_{M'}(P)$ to $\cX^R_{M}(P)$. Then we will show that these two maps are well-defined and inverse to each other.
The maps will be functorial in $P$ by construction, and so we will obtain an isomorphism $F_{M_1 \oplus_p M_2} \cong F_{M_1}\otimes F_{M_2}$.

Let $\cH$ be a modular system of $P$-hyperplane functions of $M$.
We define $\Phi(\cH)= \cH|_{E'}$. 
Now let $\cH'$ be a modular system of $P$-hyperplane functions of $M'$.
We define $\Psi$ by extending the functions in $\cH'$ to $p$.
If $f_{H}$ is in $\cH'$ and $H$ contains $E_1'$ or $E_2'$, then $H\cup p$ is a hyperplane of $M$, so we define $f_{H \cup p}(p)=0$.
Otherwise, $H$ is also a hyperplane of $M$ by \autoref{lem: hyperplanes of M'}, and we will extend $f_H$ to $p$ with the help of a fixed hyperplane $H_0$ of $M'$ that does not contain $E_1'$ or $E_2'$.
(To see that $H_0$ exists, for each $i \in \{1,2\}$, let $H_i$ be a hyperplane of $M_i$ that does not contain $p$.
Then $H_1 \cup H_2$ is a hyperplane of $M$ by \autoref{prop:hyperplanes of the parallel connection}, and therefore $H_1 \cup H_2$ is also a hyperplane of $M'$ by \autoref{lem: hyperplanes of M'}~(1).
Let $H_0 = H_1 \cup H_2$.)
Our definition of $\Psi$ will rely on the following observations, which we will use freely throughout the remainder of the proof:
\begin{itemize}
    \item If $H$ and $K$ are hyperplanes of $M'$ so that $H \cap E_i = K \cap E_i$ for some $i \in \{1,2\}$ and $f_H, f_K \in \cH'$, then $f_H|_{E_i}$ and $f_K|_{E_i}$ are scalar multiples of each other.

    \item If $H$ is a hyperplane of $M'$ that does not contain $E_1'$ or $E_2'$, then $(H \cap E_1') \cup (H_0 \cap E_2')$ is a hyperplane of $M'$.
\end{itemize}
The first follows from \autoref{prop: scaling of induced hyperplane functions}, and the second follows from \autoref{prop:hyperplanes of the parallel connection} and \autoref{lem: hyperplanes of M'}.
From these two observations, if $H$ is a hyperplane of $M'$ that does not contain $E_1'$ or $E_2'$, then $f_{(H \cap E_1') \cup (H_0 \cap E_2')}|_{E_1'}$ is a scalar multiple of $f_H|_{E_1'}$ and $f_{(H \cap E_1') \cup (H_0 \cap E_2')}|_{E_2'}$ is a scalar multiple of $f_{H_0}|_{E_2'}$, and it follows that 
there is a unique $c\in P^\times$ such that $f_{(H \cap E_1') \cup (H_0 \cap E_2')}$ is a scalar multiple of $f_{H}|_{E_1'} * (c \cdot f_{H_0}|_{E_2'})$.
We define $f_{H}(p) = c$, which completes the definition of $\Psi(\cH')$.
Before proving that $\Psi(\cH')$ is a modular system of $P$-hyperplane functions for $M'$, we will show that this definition is symmetric in $E_1'$ and $E_2'$.
To do so, we first prove the following technical claim.

\begin{claim} \label{claim: 2-sum rescaling}
Let $K$ and $L$ be hyperplanes of $M'$ so that neither contains $E_1'$ or $E_2'$ and $K \cap E_2' = L \cap E_2'$.
Let $K'$ and $L'$ be hyperplanes of $M'$ so that $K' \cap E_1' = K \cap E_1'$ and $L' \cap E_1' = L \cap E_1'$, and $K' \cap E_2' = L' \cap E_2'$.
Let $g_K$ and $g_L$ be scalar multiples of $f_K, f_L \in \cH'$, respectively, so that $g_{K}|_{E_2'} = g_{L}|_{E_2'}$. 
Then, for any scalar multiples $g_{K'}$ and $g_{L'}$ of $f_{K'},f_{L'} \in \cH'$, respectively, with $g_{K'}|_{E_1'} = g_{K}|_{E_1'}$ and $g_{L'}|_{E_1'} = g_{L}|_{E_1'}$, we have $g_{K'}|_{E_2'} = g_{L'}|_{E_2'}$.
\end{claim}
\begin{proof}
    Fix $L$, and suppose that the claim is false for $L$.
    Choose $K$ so that the claim is false for $L$ and $K$, and $r_{M'}(K \cap L)$ is maximal with this property.
    Since $K \cap E_2' = L \cap E_2'$, this is equivalent to the maximality of $r_{M'}(K \cap L \cap E_1')$.
    Assume we are given $K'$, $L'$, $g_{K'}$, and $g_{L'}$.
    If $K \cap E_1' = L \cap E_1'$, then $K = L = K' = L'$ and the result holds.
    So $K \cap E_1' \ne L \cap E_1'$.
    It follows from \autoref{lem: hyperplanes of M'} that $K$ and $L$ are also hyperplanes of $M$.
    Let $\mathscr P$ be the linear subclass of hyperplanes of $M$ that contain $p$.
    By \autoref{prop: linear subclasses} with $(H, K, \mathscr H) = (K, L, \mathscr P)$, there is a hyperplane $H$ of $M$ (possibly $H = L$) so that 
    $(K, H)$ is a modular pair, $p \notin H$, $r_{M}(H \cap L) > r_M(K \cap L)$, and $(K \cap L) \subseteq H$.
    Since $p \notin H$, \autoref{lem: hyperplanes of M'} implies that $H$ is also a hyperplane of $M'$.
    Then since $p \notin H$ and $K \cap L$ contains $L \cap E_2'$ which is a hyperplane of $M_2$, we see that $H \cap E_2' = L \cap E_2' = K \cap E_2'$.

    Let $g_H$ be the scalar multiple of $f_H$ so that $g_H|_{E_2'} = g_K|_{E_2'} = g_L|_{E_2'}$.
    Define $H'$ to be the hyperplane of $M'$ with $H' \cap E_1' = H \cap E_1'$ and $H' \cap E_2' = K' \cap E_2' = L' \cap E_2'$.
    Let $g_{H'}$ be the scalar multiple of $f_{H'}$ so that $g_{H'}|_{E_1'} = g_{H}|_{E_1'}$.
    Since $r_{M'}(H \cap L) > r_{M'}(K \cap L)$, by the maximality of $r_{M'}(K \cap L)$ we know that the claim is true for $H$ and $L$, and so $g_{H'}|_{E_2'} = g_{L'}|_{E_2'}$.
    We will complete the proof by showing that $g_{K'}|_{E_2'} = g_{H'}|_{E_2}$.
    Let $X_1 = K \cap H \cap E_1'$, so $X_1$ is a corank-2 flat of $M_1'$.
    Let $X = [\cl_{M_1}(X \cup p) \cup E_2] - p$.
    By \autoref{prop:hyperplanes of the parallel connection} and \autoref{lem: hyperplanes of M'}, $X$ is a hyperplane of $M'$.
    Moreover, $(K, H, X)$ is a modular triple of hyperplanes of $M'$, so there are constants $c,c'$ so that 
    $$g_K(e) + c \cdot g_H(e) + c'\cdot f_X(e) = 0$$
    for all $e \in E'$.
    Since $g_K|_{E_2'} = g_H|_{E_2'}$ and $f_X(e) = 0$ for all $e \in E_2'$, we see that $c = -1$, and so 
    $$g_K(e) - g_H(e) + c'\cdot f_X(e) = 0$$
    for all $e \in E'$.
    
    Next, note that $(K', H', X)$ is also a modular triple of $M'$, because $K' \cap H' \cap X$ is the union of $K \cap H \cap E_1'$ and $K' \cap E_2$, which is a corank-2 flat of $M'$.
    So there are constants $d, d'$ so that 
    $$g_{K'}(e) + d \cdot g_{H'}(e) + d'\cdot f_X(e) = 0$$
    for all $e \in E'$.
    Let $a \in (H - (K \cup X)) \cap E_1'$, and note that $a \in (H' - (K' \cup X)) \cap E_1'$ because $H|E_1' = H'|E_1'$ and $K|E_1' = K'|E_1'$.
    By plugging in $a$ to both equations, we see that $g_K(a) + c' \cdot f_X(a) = 0$ and $g_{K'}(a) + d'\cdot f_X(a) = 0$.
    Since $g_{K}(a) = g_{K'}(a)$ because $g_K|_{E_1'} = g_{K'}|_{E_1'}$, it follows that $c' = d'$.

    Now let $b \in (K - (H \cup X)) \cap E_1'$, and note that $b \in (K' - (H' \cup X)) \cap E_1'$ because $H|E_1' = H'|E_1'$ and $K|E_1' = K'|E_1'$.
    By plugging in $b$ to both equations we see that $-g_H(b) + c' \cdot f_X(b) = 0$ and $d \cdot g_{H'}(b) + d' \cdot f_X(b) = 0$.
    Since $c' = d'$ and $g_{H'}(b) = g_H(b)$ because $g_H|_{E_1'} = g_{H'}|_{E_1'}$, it follows that $d = -1$.
    Since $d = -1$, for any $e \in E_2' - (H' \cup K')$ we have $g_{K'}(e) - g_{H'}(e) = 0$, and so $g_{K'}|_{E_2'} = g_{H'}|_{E_2'}$, as desired.
\end{proof}

We have the following corollary, which is the only application of \autoref{claim: 2-sum rescaling} that we will need.
It shows that the map $\Psi$ from $\cH'$ to $\cH$ does not depend on whether we restrict $H_0$ to $E_1'$ or to $E_2'$.

\begin{claim} \label{claim: 2-sum symmetry}
Let $H$ be a hyperplane of $M'$ that contains neither $E_1'$ nor $E_2'$.
If $f_{H}|_{E_1'} * (c \cdot f_{H_0}|_{E_2'})$ is in $\cH'$ for some scalar $c$, then a scalar multiple of $(c\cdot f_{H_0}|_{E_1'}) * f_{H}|_{E_2'}$ is also in $\cH'$.
\end{claim}
\begin{proof}
Let $K$ be the hyperplane of $M'$ with $K \cap E_1' = H \cap E_1'$ and $K \cap E_2' = H_0 \cap E_2'$.
Note that $f_K = f_{H}|_{E_1'} * (c \cdot f_{H_0}|_{E_2'})$ by assumption.
Let $L = H_0$ and $K' = H$, and let $L'$ be the hyperplane of $M'$ with $L' \cap E_1' = H_0' \cap E_1'$ and $L' \cap E_2' = H \cap E_2'$.
Note that $f_K|_{E_1'} = f_{K'}|_{E_1'}$ because $K' = H$.
Let $g_K = f_K$ and $g_L = c \cdot f_L$; then $g_K|_{E_2'} = g_L|_{E_2'} = c \cdot f_{H_0}|_{E_2'}$.
Let $g_{K'}$ and $g_{L'}$ be scalar multiples of $f_{K'}, f_{L'} \in \cH'$, respectively, 
so that $g_{K'}|_{E_1'} = g_{K}|_{E_1'}$ and $g_{L'}|_{E_1'} = g_{L}|_{E_1'}$.
Then 
$$g_{K'}|_{E_1'} = g_K|_{E_1'} = f_K|_{E_1'} = f_{K'}|_{E_1'},$$
and since $g_{K'}$ is a scalar multiple of $f_{K'}$ it follows that $g_{K'} = f_{K'}$.
By applying \autoref{claim: 2-sum rescaling}, we know that $g_{K'}|_{E_2'} = g_{L'}|_{E_2'}$.
Then 
$$g_{L'}|_{E_1'} = g_{L}|_{E_1'} = c \cdot f_L|_{E_1'} = c\cdot f_{H_0}|_{E_1'}$$
and
$$g_{L'}|_{E_2'} = g_{K'}|_{E_2'} = f_{K'}|_{E_2'} = f_H|_{E_2'},$$
and so $g_{L'} = (c\cdot f_{H_0}|_{E_1'}) * f_{H}|_{E_2'}$ and the claim holds.
\end{proof}

Next, we will show that $\Psi(\cH')$ is a modular system of $P$-hyperplane functions of $M$.
Let $F$ be a corank-2 flat of $M$, and let $(H_1, H_2, H_3)$ be a modular triple of hyperplanes of $M$ so that $H_1 \cap H_2 \cap H_3 = F$.
By \autoref{lem: hyperplanes of M'}, $(H_1-p, H_2-p, H_3-p)$ is a modular triple of hyperplanes of $M'$, so there are constants $c_1, c_2, c_3$ so that 
$$c_1\cdot f_{H_1}(e) +c_2\cdot f_{H_2}(e)+c_3\cdot f_{H_3}(e) = 0$$
for all $e\in E'$.
We need only show that this also holds for $e = p$.
We consider two cases.

\textbf{Case 1:} Suppose $p \in F$.
Then $f_{H_i}(p) = 0$ for $i=1,2,3$.

\textbf{Case 2:} Suppose $p \notin F$.
Then outcome (3) of \autoref{prop: corank-2 flats of the parallel connection} holds for $F$, so there is some $i \in \{1,2\}$ so that $F \cap E_i$ is a hyperplane of $M_i$ and $F \cap E_{3-i}$ is a corank-2 flat of $M_{3-i}$.
We consider two subcases.

First suppose that $p \notin H_1 \cup H_2 \cup H_3$.
Then $H_1$, $H_2$, and $H_3$ all have the same restriction to $E_i'$, and so $f_{H_1}|_{E_i'}$, $f_{H_2}|_{E_i'}$, and $f_{H_3}|_{E_i'}$ are scalar multiples of each other.
If $i = 1$, then $H_1$, $H_2$, and $H_3$ agree on $E_1'$, so $(f_{H_1}(p),f_{H_2}(p), f_{H_3}(p))$ is a scalar multiple of $(f_{H_1}(e),f_{H_2}(e), f_{H_3}(e))$ for any $e\in E_1' -  F$. 
Hence $c_1\cdot f_{H_1}(p) +c_2\cdot f_{H_2}(p)+c_3\cdot f_{H_3}(p) =0$.
If $i = 2$, then $H_1$, $H_2$, and $H_3$ agree on $E_2'$, and it follows from \autoref{claim: 2-sum symmetry} that $(f_{H_1}(p),f_{H_2}(p), f_{H_3}(p))$ is a scalar multiple of $(f_{H_1}(e),f_{H_2}(e), f_{H_3}(e))$ for any $e\in E_2' -  F$. 
Again, $c_1\cdot f_{H_1}(p) +c_2\cdot f_{H_2}(p)+c_3\cdot f_{H_3}(p) =0$.

In the second subcase, suppose that $p \in H_j$ for some $j \in \{1,2,3\}$.
We may assume that $j = 1$.
Then $H_1$ contains $E_i$, so $f_{H_1}|_{E_i} = 0$ and we have $f_{H_3}|_{E_i'} = -\frac{c_2}{c_3} \cdot f_{H_2}|_{E_i'}$. 
First suppose that $i = 1$.
Then by the definition of $f_{H_2}(p)$, a multiple of $f_{H_2}|_{E_1'} * (f_{H_2}(p) \cdot f_{H_0}|_{E_2'})$ is in $\cH'$.
Similarly, a multiple of $f_{H_3}|_{E_1'} * (f_{H_3}(p) \cdot f_{H_0}|_{E_2'})$ is in $\cH'$.
Since $f_{H_3}|_{E_1'} = -\frac{c_2}{c_3} \cdot f_{H_2}|_{E_1'}$, a multiple of 
$-\frac{c_2}{c_3} \cdot f_{H_2}|_{E_1'} * (f_{H_3}(p) \cdot f_{H_0}|_{E_2'})$ is in $\cH'$, and by scaling we see that a multiple of $f_{H_2}|_{E_1'} * (-\frac{c_3}{c_2} \cdot f_{H_3}(p) \cdot f_{H_0}|_{E_2'})$ is in $\cH'$.
Therefore $f_{H_2}(p) = -\frac{c_3}{c_2} \cdot f_{H_3}(p)$, so $f_{H_3}(p) = -\frac{c_2}{c_3} \cdot f_{H_2}(p)$, and when $e = p$ we have 
\begin{align*}
0+c_2\cdot f_{H_2}(p) + c_3 \cdot \bigg(-\frac{c_2}{c_3} \cdot f_{H_2}(p)\bigg) = 0,
\end{align*}
as desired.
If $i = 2$, then \autoref{claim: 2-sum symmetry} allows us to use an identical argument, which we briefly describe. 
First, by the definition of $f_{H_2}(p)$ and \autoref{claim: 2-sum symmetry}, a multiple of $(f_{H_2}(p) \cdot f_{H_0}|_{E_1'}) * f_{H_2}|_{E_2'}$ is in $\cH'$.
Similarly, a multiple of $(f_{H_3}(p) \cdot f_{H_0}|_{E_1'}) * f_{H_3}|_{E_2'}$ is in $\cH'$.
Since $f_{H_3}|_{E_2'} = -\frac{c_2}{c_3} \cdot f_{H_2}|_{E_2'}$, a multiple of $(f_{H_3}(p) \cdot f_{H_0}|_{E_1'}) * (-\frac{c_2}{c_3} \cdot f_{H_2}|_{E_2'})$ is in $\cH'$.
Once again, it follows that $f_{H_3}(p) = -\frac{c_2}{c_3} \cdot f_{H_2}(p)$, and so $c_1\cdot f_{H_1}(p) +c_2\cdot f_{H_2}(p)+c_3\cdot f_{H_3}(p) = 0$, as desired.
It follows from Cases 1 and 2 that $\Psi(\cH')$ is a modular system of hyperplane functions, as claimed.

Next we will show that $\Phi$ and $\Psi$ are inverses of one another.
It is clear that $\Phi \circ \Psi$ is the identity map regardless of the choice of $H_0$. 
In the case of $\Psi \circ \Phi$, let $H_0$ be the hyperplane  that we fixed.
Note that $H_0$ is also a hyperplane of $M$.
Let $f_{H_0}\in \cH$, and let $f_{H}\in \cH$ for an arbitrary hyperplane $H$ of $M$.
Let $\overline{f_{H}}$ be the function in $\Psi \circ \Phi(\cH)$ such that $\overline{f_{H}} (e) = f_{H}(e)$ for all $e\in E'$.
If $p \in H$ then $\overline {f_H} = f_H$.
If $p \notin H$, then let $K = (H \cap E_1)\cup (H_0 \cap E_2)$; by \autoref{prop:hyperplanes of the parallel connection}, we know that $K$ is a hyperplane of $M$.
Since $K \cap E_2 = H_0\cap E_2$ we may assume, by scaling $f_K \in \cH$, that $f_{K}|_{E_2} = f_{H_0}|_{E_2}$.
In particular, $f_K(p) = f_{H_0}(p)$.
Since $K\cap E_1 = H\cap E_1$, we know that $f_{K}|_{E_1}$ is a scalar multiple of $f_{H}|_{E_1}$, and in particular we have $f_{K}|_{E_1} = \frac{f_{K}(p)}{f_{H}(p)} \cdot f_{H}|_{E_1} = \frac{f_{H_0}(p)}{f_{H}(p)} \cdot f_{H}|_{E_1}$.
Then $f_K = \left(\frac{f_{H_0}(p)}{f_{H}(p)} \cdot f_{H}|_{E_1}\right) * f_{H_0}|_{E_2}$. 
So, by definition, $\overline{f_{H}}(p) = \frac{1}{f_{H_0}(p)} \cdot f_{H}(p)$. 
The constant $\frac{1}{f_{H_0}(p)}$ only depends on the hyperplane $H_0$, so $\cH$ and $\Psi\circ\Phi(\cH)$  are in the same rescaling class. 
\end{proof}

\section{The foundation of a segment-cosegment exchange}
\label{sec:segmentcosegmentexchange}

In this section we show that if $M$ is a matroid and $X \subseteq E(M)$ is a coindependent set such that $M|X \cong U_{2,n}$ for some $n \ge 2$, then the segment-cosegment exchange of $M$ along $X$ has the same foundation as $M$.
We first recall the relevant definitions, which first appeared in \cite{Oxley-Semple-Vertigan}.

For each integer $n \ge 2$, the matroid $\Theta_n$ has ground set $X \sqcup Y$ where $X = \{x_1, x_2, \dots, x_n\}$ and $Y = \{y_1, y_2, \dots, y_n\}$, and the following bases:
\begin{itemize}
    \item $Y$, 
    
    \item $(Y - y_i) \cup x_j$ for distinct $i,j \in [n]$, and

    \item $(Y - Y') \cup X'$ where $Y' \subseteq Y$ and $X' \subseteq X$ and $|Y'| = |X'| = 2$.
\end{itemize}
The set $X$ is a modular flat of $\Theta_n$ and $\Theta_n|X \cong U_{2,n}$. 
Therefore, if $M$ is any matroid with $M|X \cong U_{2,n}$, the generalized parallel connection $P_X(M, \Theta_n)$ is well-defined.

The matroid $P_X(M, \Theta_n) \backslash X$, often denoted $\Delta_X(M)$, is called the \emph{segment-cosegment exchange} of $M$ along $X$.
When $n = 2$, $\{x_i, y_i\}$ is a series pair of $P_X(M, \Theta_2)$ for $i = 1,2$,  so $P_X(M, \Theta_2) \del X \cong M$.
When $n = 3$ we have $\Theta_3 \cong M(K_4)$ (the cycle matroid of the graph $K_4$), and $P_X(M, \Theta_3) \del X$ is also called the \emph{Delta-Wye exchange} of $M$ along $X$ \cite{Akkari-Oxley91}.

We next state some properties of $\Theta_n$.
There are three different types of hyperplanes of $\Theta_n$, depending on the size of their intersection with $X$.
This is straightforward to prove using the above description of the bases of $\Theta_n$.

\begin{prop} \label{prop: theta hyperplanes}
If $H$ is a hyperplane of $\Theta_n$, then either
\begin{enumerate}
\item $H = (Y - y_i) \cup x_i$ for some $i \in [n]$, or

\item $H = (Y - \{y_i, y_j\}) \cup x_k$ for distinct $i,j,k \in [n]$, or

\item $H = (X \cup Y) - \{y_i, y_j, y_k\}$ for distinct $i,j,k \in [n]$.
\end{enumerate}
\end{prop}

Using the previous proposition, it is straightforward to show that there are four types of corank-2 flats of $\Theta_n$.
Note that outcomes $(1)$ and $(2)$ only occur when $n \ge 4$.

\begin{prop} \label{prop: theta corank-2 flats}
If $F$ is a corank-2 flat of $\Theta_n$, then either
\begin{enumerate}
\item $F = (X \cup Y) - \{y_i, y_j, y_k, y_l\}$ for distinct $i,j,k,l \in [n]$, or

\item $F = (Y - \{y_i, y_j, y_k\}) \cup x_l$ for distinct $i,j,k, l \in [n]$, or

\item $F = (Y - \{y_i, y_j, y_k\}) \cup x_i$ for distinct $i,j,k \in [n]$, or

\item $F = Y - \{y_i, y_j\}$ for distinct $i,j \in [n]$.
\end{enumerate}
\end{prop}

We next turn our attention to representations of $U_{2,n}$, and prove two properties that hold for any modular system of hyperplane functions of $U_{2,n}$.

\begin{prop} \label{prop: line hyperplane functions}
    Let $P$ be a pasture, and let $\cH$ be a modular system of $P$-hyperplane functions for $U_{2,n}$ on the ground set $X = \{x_1, x_2, \dots, x_n\}$. Then
    \begin{enumerate}
    \setlength\itemsep{0.75em}
        \item $f_{x_i}(x_j) = -f_{x_j}(x_i)$ for all distinct $i,j, \in [n]$, and 

        \item for all $1 \le i < j < k \le n$ we have 
    $$f_{x_j}(x_k) \cdot f_{x_i}(e) + f_{x_k}(x_i) \cdot f_{x_j}(e) + f_{x_i}(x_j) \cdot f_{x_k}(e) = 0
    $$
    for all $e \in X$.
    \end{enumerate}
\end{prop}
\begin{proof}
It follows from \cite[Theorem 2.16]{Baker-Lorscheid25} that the function $\Delta \colon X^2 \to P$ defined by $\Delta(x_ix_j) = f_{x_i}(x_j)$ is a 
(weak) Grassmann-Pl\"ucker function, which implies that $(1)$ and $(2)$ hold.
\end{proof}

Finally, we need a general lemma about rescaling a modular system of hyperplane functions along a triangle.

\begin{lemma} \label{lem:rescale on a triangle}
    Let $M$ be a matroid, let $T = \{x,y,z\}$ be a triangle of $M$, and let $P$ be a pasture.
    Let $\cH$ be a modular system of $P$-hyperplane functions for $M$.
    Then there is a modular system $\cH'$ of $P$-hyperplane functions for $M$ that is rescaling equivalent to $\cH$ and has the following properties:
    \begin{enumerate}
        \item If $H$ is a hyperplane of $M$ so that $|H\cap T| = 1$, then $f_H \in \cH'$ has values $0$, $1$, and $-1$ on $T$.

        \item If $H$ is a hyperplane of $M$ disjoint from $T$, then $f_H \in \cH'$ satisfies $f_H(x) + f_H(y) + f_H(z) = 0$.
    \end{enumerate}
\end{lemma}
\begin{proof}
Let $B$ be a basis of $M/T$, and let $L = \cl_M(B)$.
Let $H_x$, $H_y$, and $H_z$ be $\cl(L \cup x)$, $\cl(L \cup y)$, and $\cl(L \cup z)$, respectively. 
Note that $(H_x, H_y, H_z)$ is a modular triple of hyperplanes of $M$.
By scaling functions in $\cH$, we may assume that if $H$ is a hyperplane and $H \cap T = \{x\}$, then $f_H(y) = 1$.
Similarly, we may assume that if $H \cap T = \{y\}$ then $f_H(z) = 1$, and if $H \cap T = \{z\}$ then $f_H(x) = 1$.
Now, scale $\cH$ by $\frac{-1}{f_{H_x}(z)}$ at $z$ and by $\frac{-1}{f_{H_y}(x)}$ at $x$, and let $\cH'$ be the resulting system of $P$-hyperplane functions for $M$.
Note that $f_{H_x}(z) = -1$ and $f_{H_y}(x) = -1$, as desired.

We first show that $f_{H_z}(y) = -1$.
Since $(H_x, H_y, H_z)$ is a modular triple, there are constants $c', c''$ so that 
$$f_{H_x}(e) + c'\cdot f_{H_y}(e) + c''\cdot f_{H_z}(e) = 0$$ for all $e \in E(M)$.
Setting $e = z$ shows that $c' = 1$, and setting $e = x$ shows that $c'' = 1$.
Then setting $e = y$ shows that $f_{H_z}(y) = -1$, as desired.

Now we prove $(1)$.
We present the argument only for hyperplanes $H$ with $H \cap T = \{x\}$, but the argument is very similar when $H \cap T \in \{y, z\}$.
Suppose there is a hyperplane $H$ of $M$ with $H \cap T  = \{x\}$ so that $f_H(z) \ne -1$, and let $r(H \cap H_x)$ be maximal with these properties.
Let $\mathscr T$ be the linear subclass of hyperplanes of $M$ that contain $T$.
By \autoref{prop: linear subclasses} with $(H, K, \mathscr H) = (H, H_x, \mathscr T)$, there is a hyperplane $H'$ (possibly $H_x$) so that $(H, H')$ is a modular pair, $H'$ contains $H \cap H_x$ but not $T$, and $r(H' \cap H_x) > r(H \cap H_x)$.
Since $H'$ contains $H \cap H_x$ but not $T$ we see that $H' \cap T = \{x\}$.
By the maximality of $r_M(H \cap H_x)$, it follows that $f_{H'}(z) = -1$.
Let $F = H \cap H'$, and let $H'' = \cl(F \cup T)$.
Then $(H, H', H'')$ is a modular triple because $F$ is a corank-2 flat of $M$, so there are constants $c, c''$ so that 
$$c\cdot f_H(e) + f_{H'}(e) + c''\cdot f_{H''}(e) = 0$$ 
for all $e \in E(M)$.
Setting $e = y$ shows that $c = -1$, and then setting $e = z$ shows that $f_H(z) = -1$, a contradiction.
This establishes $(1)$.

We now prove $(2)$.
Let $H$ be a hyperplane of $M$ which is disjoint from $T$.
Let $F$ be a corank-2 flat of $M$ contained in $H$, and let $H_x = \cl(F \cup x)$, $H_y = \cl(F \cup y )$.
Then $(H, H_x, H_y)$ is a modular triple, so there are constants $c$ and $c'$ so that
$$f_H(e) + c\cdot f_{H_x}(e) + c' \cdot f_{H_y}(e) = 0$$
for all $e \in E(M)$.
By setting $e = x$, we see that $c' = -\frac{f_H(x)}{f_{H_y}(x)}$, and by setting $e = y$, we see that $c = -\frac{f_H(y)}{f_{H_x}(y)}$.
Setting $e = z$ then gives 
$$f_H(z) -\frac{f_H(y)}{f_{H_x}(y)}\cdot f_{H_x}(z) -\frac{f_H(x)}{f_{H_y}(x)} \cdot f_{H_y}(z) = 0,$$
and since $\frac{f_{H_x}(z)}{f_{H_x}(y)} = \frac{f_{H_y}(z)}{f_{H_y}(x)} = -1$ by $(1)$, this simplifies to $f_H(z) + f_H(y) + f_H(x) = 0$.
\end{proof}

We now prove that forming the generalized parallel connection with $\Theta_n$ preserves foundations.
Note that we do not require $X$ to be coindependent; that is only necessary for the subsequent argument in which we delete $X$.

\begin{thm}
\label{thm:parallelconnectionwithTheta}
    Let $M_1$ be a matroid, let $X\subseteq E(M_1)$ so that $M_1|X \cong U_{2,n}$ for some $n \ge 2$, and let $M = P_X(M_1, \Theta_n)$.
    Then $F_M \cong F_{M_1}$.
\end{thm}
\begin{proof}
When $n = 2$ we know that the cosimplification of $M$ is isomorphic to $M_1$ because $\{x_i, y_i\}$ is a series pair of $M$ for $i = 1,2$.
So by \cite[Corollary 4.10]{Baker-Lorscheid25}, we may assume that $n \ge 3$.
Let $E_1$ be the ground set of $M_1$, and let $E_2 = X \cup Y$ be the ground set of $\Theta_n$ with $X = \{x_1, x_2, \dots, x_n\}$ and $Y = \{y_1, y_2, \dots, y_n\}$.
Let $E = E_1 \cup E_2$.

Let $P$ be a pasture.
Given a modular system $\cH$ of $P$-hyperplane functions for $M$, we define a modular system $\cH_1$ of $P$-hyperplane functions for $M_1$ by restriction to $E_1$, so $\cH_1 = \cH|_{E_1}$.
Conversely, let $\cH_1$ be a modular system of $P$-hyperplane functions for $M_1$.
Note that $\cH_1$ induces a modular system $\mathcal H_1|_X$ of $P$-hyperplane functions of $U_{2,n}$ by restriction to $X$; we write $f_{x_i}$ for the function in $\cH_1|_X$ corresponding to the hyperplane $x_i$ of $M_1|X$.
By \autoref{prop: scaling of induced hyperplane functions} we may assume, by rescaling the functions in $\cH_1$, that for all distinct $i,j \in [n]$, if $H_1$ is a hyperplane of $M_1$ with $H_1 \cap X = \{x_i\}$, then $f_{H_1}(x_j) = f_{x_i}(x_j)$.
We will define a modular system $\cH$ of $P$-hyperplane functions for $M$ so that $\cH|_{E_1} = \cH_1$, up to rescaling equivalence.

For each hyperplane $H$ of $M$, we will define the corresponding function $f_H \in \cH$ by separately considering the five different possibilities for the type of $H$.
These five possibilities arise by applying Propositions \ref{Prop:hyperplanes of the parallel connection (r(T) = 2)}, \ref{prop: theta hyperplanes}, and \ref{prop: theta corank-2 flats}; note that we split outcome (3) of \autoref{Prop:hyperplanes of the parallel connection (r(T) = 2)} into two separate cases depending on the form of the hyperplane of $\Theta_n$:
\begin{enumerate}
\setlength\itemsep{1em}
\item If $H = E_1 \cup (Y - \{y_i,y_j, y_k\})$ for distinct $i,j,k \in [n]$ with $i < j < k$, define 
\begin{itemize}
    \item $f_H(y_i) = f_{x_j}(x_k)$,

    \item $f_H(y_j) = f_{x_k}(x_i)$, and

    \item $f_H(y_k) = f_{x_i}(x_j)$.
\end{itemize}

\item If $H = H_1 \cup E_2$, where $H_1$ is a hyperplane of $M_1$ that contains $X$, define $f_H(e) = f_{H_1}(e)$ for all $e \in E$.

\item If $H = H_1 \cup ((Y - y_i) \cup x_i)$ for $i \in [n]$, where $H_1$ is a hyperplane of $M_1$ with $H_1 \cap X = \{x_i\}$, define 
\begin{itemize}
    \item $f_H(e) = f_{H_1}(e)$ for all $e \in E_1$ (in particular, $f_H(x_j) = f_{x_i}(x_j)$ for all distinct $i,j \in [n]$), and

    \item $f_H(y_i) = 1$.
\end{itemize} 

\item[(3')] If $H = H_1 \cup ((Y - \{y_i, y_j\}) \cup x_k)$ for distinct $i,j,k \in [n]$ with $i < j$, where $H_1$ is a hyperplane of $M_1$ with $H_1 \cap X = \{x_k\}$, define 
\begin{itemize}
\item $f_H(e) = f_{H_1}(e)$ for all $e \in E_1$ (in particular, $f_H(x_l) = f_{x_k}(x_l)$ for all $l \notin \{i,j,k\}$), 

\item $f_H(y_i) = \frac{-f_{x_j}(x_k)}{f_{x_i}(x_j)}$, and 

\item$f_H(y_j) = \frac{f_{x_i}(x_k)}{f_{x_i}(x_j)}$.
\end{itemize}

\item If $H = H_1 \cup (Y - \{y_i, y_j\})$ for distinct $i,j \in [n]$, where $H_1$ is a hyperplane of $M_1$ disjoint from $X$, define
\begin{itemize}
\item $f_H(e) = f_{H_1}(e)$ for all $e \in E_1$, 

\item $f_H(y_i) = \frac{f_{H_1}(x_j)}{f_{x_i}(x_j)}$, and

\item $f_H(y_i) = \frac{f_{H_1}(x_i)}{f_{x_j}(x_i)}$.
\end{itemize}
\end{enumerate}

We now have a well-defined map from $\cH_1$ to a set $\cH$ of hyperplane functions for $M$.
Clearly $\cH|_{E_1} = \cH_1$, so it suffices to show that $\cH$ is a modular system.

Let $F$ be a corank-2 flat of $M$, and let $(H, H', H'')$ be a modular triple of hyperplanes of $M$ with $H \cap H' \cap H'' = F$.
By \autoref{Prop:corank-2 flats of the parallel connection (r(T) = 2)}, there are seven possibilities for $F$, which we consider separately.
(Some cases only occur when $n \ge 4$ or $n \ge 5$.)
We split outcome (4) of \autoref{Prop:corank-2 flats of the parallel connection (r(T) = 2)} into two cases depending on the form of the hyperplane of $\Theta_n$.
Also, each hyperplane or corank-2 flat of $\Theta_n$ is associated with a given subset of $[n]$; we will explicitly choose this subset without loss of generality to improve readability. 
We also choose $(H, H', H'')$ up to permutation.

\textbf{Case 1:} $F = E - \{y_1, y_2, y_3, y_4\}$.
Then $(H, H', H'') = (F \cup y_1, F \cup y_2, F \cup y_3)$.
We will show that 
$$[f_{x_1}(x_4)] \cdot f_{H}(e) - [f_{x_2}(x_4)] \cdot f_{H'}(e) + [f_{x_3}(x_4)] \cdot f_{H''}(e) = 0$$ for all $e \in E$. 
Without loss of generality, this only needs to be checked for $e = y_1$ and $e = y_4$.
When $e = y_1$, by applying (1) we have 
$$- [f_{x_2}(x_4)] \cdot f_{x_3}(x_4) + [f_{x_3}(x_4)] \cdot f_{x_2}(x_4) = 0.$$
When $e = y_4$, using (1) we have
$$[f_{x_1}(x_4)] \cdot f_{x_2}(x_3) - [f_{x_2}(x_4)] \cdot f_{x_1}(x_3) + [f_{x_3}(x_4)] \cdot f_{x_1}(x_2),$$
which is equal to $0$ by \autoref{prop: line hyperplane functions}.

\textbf{Case 2:} $F = F_1 \cup E_2$, where $F_1$ is a corank-2 flat of $M_1$ that contains $X$.
Then there is a modular triple $(H_1, H_1', H_1'')$ of hyperplanes of $M_1$ so that $(H, H', H'') = (H_1 \cup E_2, H_1' \cup E_2,  H''_1 \cup E_2)$.
So there are constants $c, c', c''$ such that 
$$c \cdot f_{H_1}(e) + c' \cdot f_{H_1'}(e) + c'' \cdot f_{H_1''}(e) = 0$$ for all $e \in E_1$, and it follows from (2) that 
$$c \cdot f_{H}(e) + c' \cdot f_{H'}(e) + c'' \cdot f_{H''}(e) = 0$$ for all $e \in E$.

\textbf{Case 3:} $F = H_1 \cup (Y - \{y_1, y_2, y_3\})$, where $H_1$ is a hyperplane of $M_1$ that contains $X$.
Then there is no modular triple of hyperplanes containing $F$, because the only hyperplanes containing $F$ are $F \cup E_1$ and $F \cup E_2$.

\textbf{Case 4:} $F = H_1 \cup ((Y - \{y_1, y_2, y_3\}) \cup x_4)$ where $H_1$ is a hyperplane of $M_1$ with $H_1 \cap X = \{x_4\}$.
There are two subcases.
In the first subcase, $(H, H', H'') = (F \cup y_1, F \cup y_2, F \cup y_3)$.
We will show that 
\renewcommand{\theequation}{\alph{equation}}
\setcounter{equation}{0}
\begin{equation} 
\label{eq:show-equal-to-zero}    
[f_{x_1}(x_4) \cdot f_{x_2}(x_3)] \cdot f_{H}(e) + [f_{x_4}(x_2) \cdot f_{x_1}(x_3)] \cdot f_{H'}(e) + [f_{x_3}(x_4) \cdot f_{x_1}(x_2)] \cdot f_{H''}(e) = 0
\end{equation}
for all $e \in E$.
When $e \in E_1$, this follows from \autoref{prop: line hyperplane functions} and the fact that $f_H(e) = f_{H'}(e) = f_{H''}(e)$ by (3').
When $e = y_1$, using (3'), the left-hand side of (\ref{eq:show-equal-to-zero}) becomes
$$[f_{x_4}(x_2) \cdot f_{x_1}(x_3)] \cdot \frac{-f_{x_3}(x_4)}{f_{x_1}(x_3)} + [f_{x_3}(x_4) \cdot f_{x_1}(x_2)] \cdot \frac{-f_{x_2}(x_4)}{f_{x_1}(x_2)},$$
which is equal to $0$ by \autoref{prop: line hyperplane functions}.

In the second subcase, $(H, H', H'') = (F \cup y_1, F \cup y_2, F \cup E_1)$.
We will show that 
$$[f_{x_3}(x_1) \cdot f_{x_2}(x_3)] \cdot f_{H}(e) + [f_{x_2}(x_3) \cdot f_{x_1}(x_3)] \cdot f_{H'}(e) + [f_{x_3}(x_4)] \cdot f_{H''}(e) = 0$$ for all $e \in E$.
When $e \in E_1$ this follows from the fact that $f_H(e) = f_{H'}(e)$ by (3').
When $e = y_1$, using (1) and (3'), we have 
$$[f_{x_3}(x_1) \cdot f_{x_2}(x_3)] \cdot \frac{-f_{x_3}(x_4)}{f_{x_1}(x_3)} + [f_{x_3}(x_4)] \cdot {f_{x_2}(x_3)} = 0,$$
and when $e = y_3$, using (1) and (3'), the left-hand side of (\ref{eq:show-equal-to-zero}) becomes
$$[f_{x_3}(x_1) \cdot f_{x_2}(x_3)] \cdot \frac{f_{x_2}(x_4)}{f_{x_2}(x_3)} + [f_{x_2}(x_3) \cdot f_{x_1}(x_3)] \cdot \frac{f_{x_1}(x_4)}{f_{x_1}(x_3)} + [f_{x_3}(x_4)] \cdot f_{x_1}(x_2),$$
which is equal to $0$ by \autoref{prop: line hyperplane functions}.

\textbf{Case 4':} $F = H_1 \cup ((Y - \{y_1, y_2, y_3\}) \cup x_1)$, where $H_1$ is a hyperplane of $M_1$ with $H_1 \cap X = \{x_1\}$.
Then $(H, H', H'') = (F \cup y_1 , F \cup \{y_2, y_3\}, F \cup E_1)$.
We will show that 
$$[f_{x_2}(x_3)] \cdot f_{H}(e) + [f_{x_3}(x_2)] \cdot f_{H'}(e) + f_{H''}(e) = 0$$ for all $e \in E$.
When $e \in E_1$, this follows from the fact that $f_H(e) = f_{H'}(e) = f_{H_1}(e)$ by (3) and (3').
When $e = y_1$, using (1) and (3), we have
$$[f_{x_3}(x_2)] \cdot 1 + f_{x_2}(x_3) = 0.$$
When $e = y_3$, using (1) and (3'), we have 
$$[f_{x_2}(x_3)] \cdot \frac{f_{x_2}(x_1)}{f_{x_2}(x_3)} + f_{x_1}(x_2) = 0.$$

\textbf{Case 5:} $F = F_1 \cup H_2$, where $F_1$ is a corank-2 flat of $M_1$, $H_2$ is a hyperplane of $\Theta_n$, and $F_1 \cap X = H_2 \cap X = \{x_1\}$.
Then $(H \cap E_1, H' \cap E_1, H'' \cap E_1)$ is a modular triple of hyperplanes of $M_1$, so there are constants $c,c',c''$ so that 
\begin{equation}
\label{eq:show-equal-to-zero-7}
c \cdot f_{H \cap E_1}(e) + c'\cdot f_{H' \cap E_1}(e) + c'' \cdot f_{H'' \cap E_1}(e) = 0
\end{equation}
for all $e \in E_1$.
By (2), (3), and (3'), this implies that $c \cdot f_H(e) + c' \cdot f_{H'}(e) + c'' \cdot f_{H''}(e) = 0$ for all $e \in E_1$, so we only need to show that this also holds for all $y_i$.
At most one of $H, H', H''$ contains $E_2$; we may assume that $H$ and $H'$ do not contain $E_2$.
We consider two cases depending on whether or not $E_2 \subseteq H''$.
First suppose that $H''$ does not contain $E_2$.
Then $x_2 \notin H \cup H' \cup H''$.
Since $H \cap X = H' \cap X = H'' \cap X = \{x_1\}$ we know that $f_{H \cap E_1}(x_i) = f_{H' \cap E_1}(x_i) = f_{H' \cap E_1}(x_i) = f_{x_1}(x_i)$ for all $i \in [n]$ due to the scaling assumption on $\cH_1$.
Then plugging in $e = x_2$ to (\ref{eq:show-equal-to-zero-7}) shows that $c + c' + c'' = 0$.
Since $H$, $H'$, $H''$ all have the same restriction to $E_2$ (namely $H_2$), either $f_H$, $f_{H'}$, $f_{H''}$ are all defined using (3) or they are all defined using (3'), and it follows from (3) or (3') that $f_H(y_i) = f_{H'}(y_i) = f_{H''}(y_i)$ for all $i \in [n]$.
Therefore $c \cdot f_H(y_i) + c' \cdot f_{H'}(y_i) + c'' \cdot f_{H''}(y_i) = 0$ for all $i \in [n]$.

In the second case, suppose that $E_2 \subseteq H''$. 
Since $H \cap X = H' \cap X = \{x_1\}$ we know that $f_{H \cap E_1}(x_i) = f_{H' \cap E_1}(x_i) = f_{x_1}(x_i)$ for all $i \in [n]$ due to the scaling assumption on $\cH_1$.
Then plugging in $e = x_2$ to (\ref{eq:show-equal-to-zero-7}) shows that $c + c' = 0$, because $x_2 \in H''$.
Since $H$ and $H'$ have the same restriction to $E_2$ (namely $H_2$), either $f_H$ and $f_{H'}$ are both defined using (3) or they are both defined using (3'), and it follows from (3) or (3') that $f_H(y_i) = f_{H'}(y_i)$ for all $i \in [n]$.
Since $f_{H''}(y_i) = 0$ for all $i \in [n]$ and $c + c' = 0$, we see that $c \cdot f_H(y_i) + c' \cdot f_{H'}(y_i) + c'' \cdot f_{H''}(y_i) = 0$ for all $i \in [n]$.

\textbf{Case 6:} $F = H_1 \cup (Y - \{y_1, y_2, y_3\})$, where $H_1$ is a hyperplane of $M_1$ disjoint from $X$.
\autoref{lem:rescale on a triangle} (1) implies that by scaling $\cH_1$ at the triangle $\{x_1, x_2, x_3\}$, we may assume that if $H_0$ is a hyperplane of $M_1$ with $|H_0 \cap \{x_1, x_2, x_3\}| = 1$ then $f_{H_0}$ takes values $0$, $1$, and $-1$ on $\{x_1, x_2, x_3\}$.
It follows from \autoref{lem:rescale on a triangle} (2) that $f_{H_1}(x_1) + f_{H_1}(x_2) + f_{H_1}(x_3) = 0$. 
We may further assume, by rescaling functions, that $f_{x_1}(x_2) = 1$.

We now consider two subcases.
In the first subcase, $(H, H', H'') = (F \cup y_1, F \cup y_2, F \cup y_3)$.
We will show that 
\begin{equation}
\label{eq:show-equal-to-zero-2}
[f_{H_1}(x_1) \cdot f_{x_2}(x_3)] \cdot f_H(e) + [f_{H_1}(x_2) \cdot f_{x_3}(x_1)] \cdot f_H(e) + [f_{H_1}(x_3) \cdot f_{x_1}(x_2)] \cdot f_H(e) = 0
\end{equation}
for all $e \in E$.

When $e \in E_1$, we know that $f_H(e) = f_{H'}(e) = f_{H''}(e) = f_{H_1}(e)$ by (4).
Since $f_{x_1}(x_2) = 1$, we know that $f_{x_1}(x_3) = -1$, and so by \autoref{prop: line hyperplane functions} we have $f_{x_3}(x_1) = 1$.
Similarly, $f_{x_2}(x_3) = 1$, and then \eqref{eq:show-equal-to-zero-2} holds because 
$f_{H_1}(x_1) + f_{H_1}(x_2) + f_{H_1}(x_3) = 0$ by \autoref{lem:rescale on a triangle} (2).

When $e = y_1$, using (4), the equation \eqref{eq:show-equal-to-zero-2} reduces to 
$$[f_{H_1}(x_2) \cdot f_{x_3}(x_1)] \cdot \frac{f_{H_1}(x_3)}{f_{x_1}(x_3)} + [f_{H_1}(x_3) \cdot f_{x_1}(x_2)] \cdot \frac{f_{H_1}(x_2)}{f_{x_1}(x_2)} = 0.$$

In the second subcase, $(H, H', H'') = (F \cup y_1, F \cup y_2, F \cup E_1)$.
It is similarly straightforward to check that 
\begin{equation}
\label{eq:show-equal-to-zero-3}
[f_{x_1}(x_3) \cdot f_{x_2}(x_3)] \cdot f_H(e) + [f_{x_3}(x_1) \cdot f_{x_2}(x_3)] \cdot f_{H'}(e) + [f_{H_1}(x_3)] \cdot f_{H''}(e) = 0
\end{equation}
for all $e \in E$.
When $e \in E_1$, this follows from the fact that $f_H(e) = f_{H'}(e) = f_{H_1}(e)$ by (4).
When $e = y_1$, applying (1) and (4) gives
$$[f_{x_3}(x_1) \cdot f_{x_2}(x_3)] \cdot \frac{f_{H_1}(x_3)}{f_{x_1}(x_3)} + [f_{H_1}(x_3)] \cdot f_{x_2}(x_3) = 0,$$
 and when $e = y_3$, applying (4) shows that the left-hand side of \eqref{eq:show-equal-to-zero-3} is equal to
$$[f_{x_1}(x_3) \cdot f_{x_2}(x_3)] \cdot \frac{f_{H_1}(x_2)}{f_{x_3}(x_2)} + [f_{x_3}(x_1) \cdot f_{x_2}(x_3)] \cdot \frac{f_{H_1}(x_1)}{f_{x_3}(x_1)} + [f_{H_1}(x_3)] \cdot f_{x_1}(x_2).$$
This is equal to $0$ because, as described in the previous subcase, $f_{x_1}(x_2) = f_{x_3}(x_1) = f_{x_2}(x_3) = 1$ and $f_{H_1}(x_1) + f_{H_1}(x_2) + f_{H_1}(x_3) = 0$.

\textbf{Case 7:} $F = F_1 \cup (Y - \{y_1, y_2\})$, where $F_1$ is a corank-2 flat of $M_1$ disjoint from $X$.
We first prove:

\begin{claim} \label{claim: triples in M_1}
Let $(H_i, H_j, H_k)$ be a modular triple of hyperplanes of $M_1$ so that $H_i \cap X = \{x_i\}$, $H_j \cap X = \{x_j\}$, and $H_k \cap X = \{x_k\}$.
Then $$[f_{x_j}(x_k)] \cdot f_{H_i}(e) - [f_{x_i}(x_k)] \cdot f_{H_j}(e)
+ [f_{x_i}(x_j)] \cdot f_{H_k}(e) = 0$$
for all $e \in E_1$.
\end{claim}
\begin{proof}
We may assume that $(i,j,k) = (1,2,3)$.
There are constants $c_1,c_2, c_3$ so that 
$$c_1 \cdot f_{H_1}(e) + c_2 \cdot f_{H_2}(e) + c_3 \cdot f_{H_3}(e) = 0$$
for all $e \in E_1$.
By plugging in $e = x_1, x_2, x_3$ and using the assumption that $H_{l} \cap X = \{x_l\}$ implies $f_{H_{l}}(x_m) = f_{x_l}(x_m)$ for all $l,m \in [n]$, we see that 
$$(c_1, c_2, c_3) = (f_{x_2}(x_3), f_{x_3}(x_1), f_{x_1}(x_2))$$ up to multiplication by a scalar. This proves the claim.
\end{proof}

We now consider three subcases.
In the first subcase, $(H, H', H'') = (F \cup \{x_1, y_2\}, F \cup \{x_2, y_1\}, F \cup x_3)$.
We will show that 
$$[f_{x_2}(x_3)] \cdot f_{H}(e) - [f_{x_1}(x_3)] \cdot f_{H'}(e) + [f_{x_1}(x_2)] \cdot f_{H''}(e) = 0$$ 
for all $e \in E$.
When $e \in E_1$, this holds by \autoref{claim: triples in M_1} with $(i,j,k) = (1,2,3)$ and $(H_i, H_j, H_k) = (H, H', H'')$.
When $e = y_1$, using (3) and (4) we have
$$[f_{x_2}(x_3)] \cdot 1 + [f_{x_1}(x_2)] \cdot \frac{-f_{x_2}(x_3)}{f_{x_1}(x_2)} = 0.$$ 

In the second subcase, $(H, H', H'') = (F \cup \{x_1, y_2\}, F \cup x_3, F \cup x_4)$.
We will show  that 
\begin{equation}
\label{eq:show-equal-to-zero-4}
[f_{x_3}(x_4)] \cdot f_{H}(e) - [f_{x_1}(x_4)] \cdot f_{H'}(e) + [f_{x_1}(x_3)]\cdot f_{H''}(e) = 0
\end{equation}
for all $e \in E$.
When $e \in E_1$, this holds by \autoref{claim: triples in M_1} with $(i,j,k) = (1,3,4)$ and $(H_i, H_j, H_k) = (H, H', H'')$.
When $e = y_2$, applying (3') shows that 
$$ - [f_{x_1}(x_4)] \cdot \frac{f_{x_1}(x_3)}{f_{x_1}(x_2)} + [f_{x_1}(x_3)]\cdot \frac{f_{x_1}(x_4)}{f_{x_1}(x_2)} = 0.$$
When $e = y_1$, by applying (3) and (3'), the left-hand side of (\ref{eq:show-equal-to-zero-4}) becomes
$$[f_{x_3}(x_4)] \cdot 1 - [f_{x_1}(x_4)] \cdot \frac{-f_{x_2}(x_3)}{f_{x_1}(x_2)} + [f_{x_1}(x_3)]\cdot \frac{-f_{x_2}(x_4)}{f_{x_1}(x_2)},$$
which is equal to $0$ by \autoref{prop: line hyperplane functions}.

\setcounter{equation}{6}

In the third subcase, $(H, H', H'') = (F \cup x_3, F \cup x_4, F \cup x_5)$. 
We will show that 
\begin{equation}
\label{eq:show-equal-to-zero-5}
[f_{x_4}(x_5)] \cdot f_{H}(e) - [f_{x_3}(x_5)] \cdot f_{H'}(e) + [f_{x_3}(x_4)]\cdot f_{H''}(e) = 0
\end{equation}
for all $e \in E$.
When $e \in E_1$ this holds by \autoref{claim: triples in M_1}
with $(i,j,k) = (3,4,5)$ and $(H_i, H_j, H_k) = (H, H', H'')$.
When $e = y_2$, using (3), the left-hand side of (\ref{eq:show-equal-to-zero-5}) becomes 
$$[f_{x_4}(x_5)] \cdot \frac{f_{x_1}(x_3)}{f_{x_1}(x_2)} - [f_{x_3}(x_5)] \cdot \frac{f_{x_1}(x_4)}{f_{x_1}(x_2)} + [f_{x_3}(x_4)]\cdot \frac{f_{x_1}(x_5)}{f_{x_1}(x_2)},$$
which is equal to $0$ by \autoref{prop: line hyperplane functions}.

These seven cases combine to show that $\cH$ is in fact a modular system of $P$-hyperplane functions for $M$.
So, for any pasture $P$, we have defined a map from $\cX^R_{M_1}(P)$ to $\cX^R_{M}(P)$.
The inverse of the map from $\cX^R_{M}(P)$ to $\cX^R_{M_1}(P)$ is the natural map defined by restriction to $E_1$, which is clearly functorial in $P$. This implies that $M_1$ and $M$ have isomorphic foundations.
\end{proof}

This has the following corollary in the special case that $M_1 \cong U_{2,n}$.

\begin{cor}
For all $n \ge 2$, the matroids $U_{2,n}$ and $\Theta_n$ have isomorphic foundations.    
\end{cor}

We next delete $X$ from $P_X(M, \Theta_n)$ and show that this preserves the foundation when $X$ is coindependent in $M$.
We will use the following lemma.

\begin{lemma} \label{lem:pasture surjection}
If $P$ is a finitely generated pasture and $f : P \to P$ is a homomorphism which restricts to a surjection $P^\times \to P^\times$ of multiplicative groups, then $f$ is an isomorphism.
\end{lemma}
\begin{proof}
A surjective homomorphism from a finitely generated abelian group to itself is necessarily an isomorphism, cf.~\cite[Proof of Lemma 29.2]{Matsumura80}.
So $f$ is a bijection on underlying sets, and by construction $f(N_P) \subseteq N_P$. It suffices to prove that the map from $P$ to $P$ which 
sends $x \in P$ to $f^{-1}(x) \in P$ is a homomorphism. 

Let $g : P \to P'$ be the homomorphism of pastures induced by the inverse map $f^{-1} : P \to P$, i.e., $P'$ has the same underlying set as $P$, and we define the null set of $P'$ to consist of all formal sums of the form $\sum a_i y_i$ such that $\sum a_i f^{-1}(y_i) \in N_P$.
Then $g \circ f : P \to P'$ is the identity map on underlying sets, and therefore $N_P \subseteq N_{P'}$.
For the reverse containment, suppose $\sum a_i y_i \in N_{P'}$. By definition, there exist $x_i \in P$ such that $f(x_i)=y_i$ and $\sum a_i x_i \in N_P$. Since $f : P \to P$ is a homomorphism, we must have $\sum a_i f(x_i) \in N_P$, which means that $N_{P'} \subseteq N_P$.
\end{proof}

We next describe the homomorphism to which we will apply \autoref{lem:pasture surjection}.
It will be defined using cross ratios; see \autoref{sec:cross ratios} for the relevant definitions.
Let $N$ be a matroid with a coindependent set $X$.
If $\cross{e_1}{e_2}{e_3}{e_4}{J}$ is a cross ratio of $N \del X$, then $\cross{e_1}{e_2}{e_3}{e_4}{J}$ is also a cross ratio of $N$.
It follows from \cite[Proposition 4.9]{Baker-Lorscheid25} that the function $\psi_{N \del X}$ from $F_{N \del X}^{\times}$ to $F_N^{\times}$ that maps $\cross{e_1}{e_2}{e_3}{e_4}{J}$ to $\cross{e_1}{e_2}{e_3}{e_4}{J}$ is a homomorphism.
We next show that in the special case that $N = P_X(M_1, \Theta_n)$ for some matroid $M_1$, this homomorphism is surjective.

\begin{lemma} \label{lem:cross ratio surjection}
    Let $M_1$ be a matroid and let $X \subseteq E(M_1)$ be a coindependent set such that $M_1|X \cong U_{2,n}$ for some $n \ge 2$. 
    Let $M = P_X(M_1, \Theta_n)$, let $M' = M \backslash X$, and let $\psi_{M \del X}$ be the homomorphism from $F_{M'}^\times$ to $F_M^\times$ that maps $\cross{e_1}{e_2}{e_3}{e_4}{J}$ to $\cross{e_1}{e_2}{e_3}{e_4}{J}$. Then $\psi_{M \del X}$ is surjective.
\end{lemma} 
\begin{proof}
Let $E$, $E_1$, and $E_2 = X \cup Y$ be the ground sets of $M$, $M_1$, and $\Theta_n$, respectively.
When $n = 2$, we know that the cosimplification of $M$ is isomorphic to $M_1$ because $\{x_i, y_i\}$ is a series pair of $M$ for $i = 1,2$.
So, by \cite[Corollary 4.10]{Baker-Lorscheid25}, we may assume that $n \ge 3$.
The following claim will allow us to show that two given cross ratios of $M$ are equal.

\begin{claim} \label{claim: identification of cross ratios}
Let $\cross{e_1}{e_2}{e_3}{e_4}{J}$ be a cross ratio of $M$.
\begin{enumerate}
\setlength\itemsep{1.0em}
    \item $\cross{e_1}{e_2}{e_3}{e_4}{J} = \cross{e_2}{e_1}{e_4}{e_3}{J} = \cross{e_3}{e_4}{e_1}{e_2}{J} = \cross{e_4}{e_3}{e_2}{e_1}{J}.$

    \item If $\cl(J) = \cl(J')$, then $\cross{e_1}{e_2}{e_3}{e_4}{J} = \cross{e_1}{e_2}{e_3}{e_4}{J'}$.
    
    \item If $\cl(J \cup e_4) = \cl(J \cup e_4')$, then $\cross{e_1}{e_2}{e_3}{e_4}{J} = \cross{e_1}{e_2}{e_3}{e_4'}{J}$. 
    
    \item If $(Ie_5; e_1, e_2, e_3, e_4)$, $(Ie_3; e_1, e_2, e_4, e_5)$, and $(Ie_4; e_1, e_2, e_5, e_3)$ are all in $\Omega_M$, then 
    \begin{align*}
    \cross{e_1}{e_2}{e_3}{e_4}{Ie_5} \cdot \cross{e_1}{e_2}{e_4}{e_5}{Ie_3} \cdot \cross{e_1}{e_2}{e_5}{e_3}{Ie_4} = 1.
    \end{align*}
\end{enumerate}
\end{claim}
\begin{proof}
Parts (1) and (4) are relations (R$\sigma$) and (R4), respectively, of \cite[Theorem 4.21]{Baker-Lorscheid25}, and parts (2) and (3) are implied by \cite[Corollary 3.7]{Baker-Lorscheid25}.
\end{proof}

Fix a cross ratio $\cross{e_1}{e_2}{e_3}{e_4}{J}$ of $M$, and let $F = \cl(J)$.
We will show that $\cross{e_1}{e_2}{e_3}{e_4}{J}$ is in the image of $\psi_{M\del X}$.
By \autoref{Prop:corank-2 flats of the parallel connection (r(T) = 2)}, there are seven possibilities for $F$, which we consider separately. 
In Cases 1--6 we will show that $\cross{e_1}{e_2}{e_3}{e_4}{J}$ is the image of a cross ratio of $M'$, and in Case 7 we will show that $\cross{e_1}{e_2}{e_3}{e_4}{J}$ is the image of a product of cross ratios of $M'$.
Each hyperplane or corank-2 flat of $\Theta_n$ is associated with a given subset of $[n]$; we will choose this subset explicitly without loss of generality to improve readability.

\textbf{Case 1:} $F = E - \{y_1, y_2, y_3, y_4\}$. 
Then $e_1, e_2, e_3, e_4 \notin X$ and $X$ is spanned in $M$ by $F - X$ because $X$ is coindependent in $M_1$.
Let $J'$ be a basis of $F - X$.
Then $\cross{e_1}{e_2}{e_3}{e_4}{J'} = \cross{e_1}{e_2}{e_3}{e_4}{J}$ by \autoref{claim: identification of cross ratios}~(2) and $J' \cup \{e_1, e_2, e_3, e_4\}$ is disjoint from $X$, so $\psi_{M \del X}$ maps $\cross{e_1}{e_2}{e_3}{e_4}{J'}$ to $\cross{e_1}{e_2}{e_3}{e_4}{J}$.

\textbf{Case 2:} $F = F_1 \cup E_2$, where $F_1$ is a corank-2 flat of $M_1$ that contains $X$.
Then $e_1, e_2, e_3, e_4 \notin X$ and $X$ is spanned in $M$ by $F - X$ because $Y \subseteq F - X$.
Let $J'$ be a basis of $F - X$.
Then $\cross{e_1}{e_2}{e_3}{e_4}{J'} = \cross{e_1}{e_2}{e_3}{e_4}{J}$ by \autoref{claim: identification of cross ratios}~(2) and $J' \cup \{e_1, e_2, e_3, e_4\}$ is disjoint from $X$, so $\psi_{M \del X}$ maps $\cross{e_1}{e_2}{e_3}{e_4}{J'}$ to $\cross{e_1}{e_2}{e_3}{e_4}{J}$.

\textbf{Case 3:} $F = H_1 \cup (Y - \{y_1, y_2, y_3\})$, where $H_1$ is a hyperplane of $M_1$ that contains $X$.
Then $M/J$ has at most two parallel classes (namely, $E_1 - F$ and $E_2 - F$), so $\cross{e_1}{e_2}{e_3}{e_4}{J}$ is degenerate and therefore $\cross{e_1}{e_2}{e_3}{e_4}{J} = 1$ in $F_M^{\times}$.

\textbf{Case 4:} $F = H_1 \cup ((Y - \{y_1, y_2, y_3\}) \cup x_i)$ for some $i \in \{1,4\}$, where $H_1$ is a hyperplane of $M_1$ with $H_1 \cap X = \{x_i\}$.
We separately consider the cases $i = 1$ and $i = 4$.

If $i = 1$, then $y_2$ and $y_3$ are parallel in $M/J$ because $(Y - y_1) \cup x_1$ is a hyperplane of $\Theta_n$.
So $M/J$ has at most three parallel classes: $E_1 - F$, $\{y_1\}$, and $\{y_2, y_3\}$.
Therefore $\cross{e_1}{e_2}{e_3}{e_4}{J}$ is degenerate, so $\cross{e_1}{e_2}{e_3}{e_4}{J} = 1$ in $F_M^{\times}$.

Suppose $i = 4$.
Then $M/J$ has at most four parallel classes: $E_1 - F$, $\{y_1\}$, $\{y_2\}$, and $\{y_3\}$.
We may assume that $\{e_1, e_2, e_3, e_4\}$ contains one element from each of these parallel classes, or else $\cross{e_1}{e_2}{e_3}{e_4}{J}$ is degenerate.
By swapping rows and columns of $\cross{e_1}{e_2}{e_3}{e_4}{J}$, we may assume that $e_4 \in E_1 - F$.
Since $X$ is coindependent in $M_1$, there is some $a \in E_1 - (H_1 \cup X)$.
Then $\cl(J \cup e_4) = \cl(J \cup a)$ because $e_4$ and $a$ are parallel in $M/J$, so by \autoref{claim: identification of cross ratios}~(3) we may assume that $e_4 = a$ and $\{e_1, e_2, e_3\} = \{y_1, y_2, y_3\}$.
Up to re-indexing, we may assume that $e_3 = y_3$, so $\cross{e_1}{e_2}{e_3}{e_4}{J} = \cross{e_1}{e_2}{y_3}{a}{J}$.
Let $I = J - y_4$.
By \autoref{claim: identification of cross ratios}~(4) we have
\begin{align*}
    \cross{e_1}{e_2}{y_3}{a}{Iy_4} \cdot \cross{e_1}{e_2}{a}{y_4}{Iy_3} \cdot \cross{e_1}{e_2}{y_4}{y_3}{Ia} = 1.
    \end{align*}
Since $(Y - y_4) \cup x_4$ is a hyperplane of $\Theta_n$ we see that $y_1$ and $y_2$ are parallel in $M/(I \cup y_3)$.
Then $M/(I \cup y_3)$ has at most three parallel classes (namely $E_1 - \cl(I \cup y_3)$, $\{y_1, y_2\}$, and $\{y_4\}$), so $\cross{e_1}{e_2}{a}{y_4}{Iy_3}$ is degenerate.
Since $\cl(I \cup a) = E - \{y_1, y_2, y_3, y_4\}$, the cross ratio $\cross{e_1}{e_2}{y_4}{y_3}{Ia}$ is the image under $\psi_{M \del X}$ of a cross ratio of $M'$, as proved in Case 1.
Since $\psi_{M\del X}$ is a homomorphism, $\cross{e_1}{e_2}{y_3}{a}{Iy_4}$ is the image under $\psi_{M \del X}$ of the inverse of $\cross{e_1}{e_2}{y_4}{y_3}{Ia}$ in $F_{M'}^{\times}$.

\textbf{Case 5:} $F = F_1 \cup H_2$, where $F_1$ is a corank-2 flat of $M_1$, $H_2$ is a hyperplane of $\Theta_n$, and $F_1 \cap X = H_2 \cap X = \{x_1\}$.
We consider two subcases, depending on the form of $H_2$.

First, suppose that $H_2 = (Y - y_1) \cup x_1$.
Then $F - x_1$ contains a basis $J'$ of $F$ because $Y - y_1$ spans $x_1$ in $\Theta_n$ and therefore in $M$ as well.
Suppose that $e_i \in X$ for some $i \in [4]$.
Since $X$ is contained in a parallel class of $M/J'$, this choice of $i$ is unique, or else $\cross{e_1}{e_2}{e_3}{e_4}{J}$ is degenerate.
By \autoref{claim: identification of cross ratios} (1) we may assume that $i = 4$.
Then $e_4$ and $y_1$ are parallel in $M/J'$, so $\cl(J' \cup e_4) = \cl(J' \cup y_1)$ and therefore $\cross{e_1}{e_2}{e_3}{e_4}{J'} = \cross{e_1}{e_2}{e_3}{y_1}{J'}$ by \autoref{claim: identification of cross ratios}~(3).
Since $J' \cup \{e_1, e_2, e_3, y_1\}$ is disjoint from $X$, we see that $\cross{e_1}{e_2}{e_3}{y_1}{J'}$ is a cross ratio of $M'$ whose image under $\psi_{M \del X}$ is $\cross{e_1}{e_2}{e_3}{e_4}{J}$.

In the second subcase, suppose that $H_2 = (Y - \{y_2, y_3\}) \cup x_1$.
Since $E_2 - H_2$ is contained in a parallel class of $M/J$, at most one of $e_1, e_2, e_3, e_4$ is in $E_2$ or else $\cross{e_1}{e_2}{e_3}{e_4}{J}$ is degenerate.
Suppose that $e_i \in E_2$ for some $i \in [4]$.
By \autoref{claim: identification of cross ratios} we may assume that $i = 4$.
Then $e_4$ and $y_2$ are parallel in $M/J$, so $\cl(J \cup e_4) = \cl(J \cup y_2)$ and therefore $\cross{e_1}{e_2}{e_3}{e_4}{J} = \cross{e_1}{e_2}{e_3}{y_2}{J}$ by \autoref{claim: identification of cross ratios}~(3).
Note that $\{e_1, e_2, e_3\} \subseteq E_1 - X$.
Let $I = J - y_1$.
By \autoref{claim: identification of cross ratios}~(4) we have
\begin{align*}
    \cross{e_1}{e_2}{e_3}{y_2}{Iy_1} \cdot \cross{e_1}{e_2}{y_2}{y_1}{Ie_3} \cdot \cross{e_1}{e_2}{y_1}{e_3}{Iy_2} = 1.
    \end{align*}
Since $(Y - y_1) \cup x_1$ is a hyperplane of $\Theta_n$ we see that $y_2$ and $y_3$ are parallel in $M/(I \cup e_3)$.
Then $M/(I \cup e_3)$ has at most three parallel classes (namely $E_1 - \cl(I \cup e_3)$, $\{y_1\}$, and $\{y_2, y_3\}$), so $\cross{e_1}{e_2}{y_2}{y_1}{Ie_3}$ is degenerate.
Since $\cl(I \cup y_2)$ is a corank-2 flat of $M_1$ consisting of $F_1$ and the hyperplane $(Y - y_1) \cup x_1$ of $\Theta_n$, we know from the first subcase of Case 5 that $\cross{e_1}{e_2}{y_1}{e_3}{Iy_2}$ is the image under $\psi_{M \del X}$ of a cross ratio of $M'$.
Since $\psi_{M \del X}$ is a homomorphism, $\cross{e_1}{e_2}{e_3}{y_2}{Iy_1}$ is the image under $\psi_{M \del X}$ of the inverse of $\cross{e_1}{e_2}{y_1}{e_3}{Iy_2}$ in $F_{M'}^{\times}$.

\textbf{Case 6:} $F = H_1 \cup (Y - \{y_1, y_2, y_3\})$, where $H_1$ is a hyperplane of $M_1$ disjoint from $X$.
Then $M/J$ has at most four nontrivial parallel classes: $E_1 - H_1$, $\{y_1\}$, $\{y_2\}$, and $\{y_3\}$.
We may assume that $\{e_1, e_2, e_3, e_4\}$ contains one element from each of these parallel classes, or else $\cross{e_1}{e_2}{e_3}{e_4}{J}$ is degenerate.
By \autoref{claim: identification of cross ratios} we may assume that $e_4 \in E_1 - H_1$ and $\{e_1, e_2, e_3\} = \{y_1, y_2, y_3\}$.
Let $a \in E_1 - (H_1 \cup X)$; such an element exists because $X$ is coindependent in $M_1$.
Then $\cl(J \cup e_4) = \cl(J \cup a)$ because $e_4$ and $a$ are parallel in $M/J$, so $\cross{e_1}{e_2}{e_3}{e_4}{J} = \cross{e_1}{e_2}{e_3}{a}{J}$ by \autoref{claim: identification of cross ratios}~(3).
Since $J \cup \{e_1, e_2, e_3, a\}$ is disjoint from $X$ we see that $\cross{e_1}{e_2}{e_3}{a}{J}$ is a cross ratio of $M'$ whose image under $\psi_{M \del X}$ is $\cross{e_1}{e_2}{e_3}{e_4}{J}$.

\textbf{Case 7:} $F = F_1 \cup (Y - \{y_1, y_2\})$, where $F_1$ is a corank-2 flat of $M_1$ disjoint from $X$.
Note that $\{x_1, y_2\}$ and $\{x_2, y_1\}$ are parallel pairs in $M/J$.
Let $k = |\{e_1, e_2, e_3, e_4\} \cap X|$.
We will proceed by induction on $k$ to show that every cross ratio $\cross{e_1}{e_2}{e_3}{e_4}{J}$ with $\cl(J) = F$ is in the image of $\psi_{M \del X}$.
If $k = 0$, then $\cross{e_1}{e_2}{e_3}{e_4}{J}$ is a cross ratio of $M'$ whose image under $\psi_{M \del X}$ is $\cross{e_1}{e_2}{e_3}{e_4}{J}$.
So we may assume that $k \ge 1$, so $e_i = x_j$ for some $i \in [4]$ and $j \in [n]$.
By \autoref{claim: identification of cross ratios} we may assume that $i = 4$.
If $j = 1$, then since $\cl(J \cup x_1) = \cl(J \cup y_2)$ because $\{x_1, y_2\}$
is a parallel pair of $M/J$, by \autoref{claim: identification of cross ratios}~(3) we see that $\cross{e_1}{e_2}{e_3}{e_4}{J} = \cross{e_1}{e_2}{e_3}{y_2}{J}$.
By induction, $\cross{e_1}{e_2}{e_3}{y_2}{J}$ is in the image of $\psi_{M \del X}$, and therefore so is $\cross{e_1}{e_2}{e_3}{e_4}{J}$.
So we may assume that $j \ne 1$, and by similar reasoning, that $j \ne 2$.
Without loss of generality, we may assume that $j = 3$, so $e_4 = x_3$.
Let $I = J - y_3$.
By \autoref{claim: identification of cross ratios}~(4) we have
\begin{align*}
    \cross{e_1}{e_2}{e_3}{x_3}{Iy_3} \cdot \cross{e_1}{e_2}{x_3}{y_3}{Ie_3} \cdot \cross{e_1}{e_2}{y_3}{e_3}{Ix_3} = 1.
    \end{align*}
Since $\cl(I \cup x_3)$ is a corank-2 flat of the form considered in Case 4, we know that $\cross{e_1}{e_2}{y_3}{e_3}{Ix_3}$ is in the image of $\psi_{M\del X}$.
We next show that $\cross{e_1}{e_2}{x_3}{y_3}{Ie_3}$ is in the image of $\psi_{M \del X}$ by considering three possibilities for $e_3$.
If $\cl(I \cup e_3) \cap X \ne \varnothing$ (in particular, if $e_3 \in X$), then $\cl(I\cup e_3)$ is a corank-2 flat of the form considered in Case 4, so $\cross{e_1}{e_2}{x_3}{y_3}{Ie_3}$ is in the image of $\psi_{M \del X}$.
If $e_3 \in \{y_1, y_2\}$, then without loss of generality we may assume that $e_3 = y_1$.
Then $y_2$ and $y_3$ are parallel in $M/J$ because of the hyperplane $(Y - y_3) \cup x_3$ of $\Theta_n$.
So $\cl(I \cup e_3 \cup x_3) = \cl(I \cup e_3 \cup y_2)$, so \autoref{claim: identification of cross ratios}~(3) implies that $\cross{e_1}{e_2}{x_3}{y_3}{Ie_3} = \cross{e_1}{e_2}{y_2}{y_3}{Ie_3}$.
By induction, $\cross{e_1}{e_2}{y_2}{y_3}{Ie_3}$ is in the image of $\psi_{M \del X}$, and therefore so is $\cross{e_1}{e_2}{x_3}{y_3}{Ie_3}$.
Finally, if $e_3 \in E_1 - X$ and $\cl(I\cup e_3)$ is disjoint from $X$, then $\cl(I \cup e_3)$ is a corank-2 flat of the form considered in Case 6 and is therefore in the image of $\psi_{M \del X}$.
Therefore, since $\psi_{M \del X}$ is a homomorphism, $\cross{e_1}{e_2}{e_3}{x_3}{Iy_3}$ is the image under $\psi_{M \del X}$ of the product of the inverses of $\cross{e_1}{e_2}{x_3}{y_3}{Ie_3}$ and $\cross{e_1}{e_2}{y_3}{e_3}{Ix_3}$ in $F_{M'}^{\times}$.
\end{proof}

We can now prove the main result of this section. 

\begin{thm} \label{thm: segment-cosegment exchange}
Let $M$ be a matroid and let $X \subseteq E(M)$ be a coindependent set so that $M|X \cong U_{2,n}$.
Then the foundation of the segment-cosegment exchange of $M$ along $X$ is isomorphic to the foundation of $M$.
\end{thm}

\begin{proof}
Following \cite{Oxley06}, for a matroid $N$ with $X \subseteq E(N)$ so that $X$ is coindependent and $N|X \cong U_{2,n}$, we write $\Delta_X(N)$ for $P_X(N, \Theta_n) \del X$, the segment-cosegment exchange of $N$ along $X$.
(We do not follow the convention from \cite{Oxley06} of relabeling $Y$ with $X$ in $\Delta_X(N)$ via the natural isomorphism from $\Theta_n$ to $\Theta_n^*$ that swaps $x_i$ and $y_i$ for each $i \in [n]$.)
Dually, if $N^*|Y \cong U_{2,n}$ then we write $\nabla_Y(N)$ for $(\Delta_Y(N^*))^*$, the cosegment-segment exchange of $N$ along $Y$.

Let $P = P_X(M, \Theta_n)$ and let $M' = P_X(M, \Theta_n)\backslash X$.
By \cite[Lemma 11.5.6]{Oxley06} we know that $((M')^*|Y) \cong U_{2, n}$, so let $P' = P_Y((M')^*, \Theta_n^*)$.
By \cite[Proposition 11.5.11 (i)]{Oxley06} we know that $\nabla_Y(\Delta_X(M)) = M$.
Taking the dual of both sides, we see that $\Delta_Y((M')^*) = M^*$, so $P' \del Y = M^*$.
It follows from \autoref{thm:parallelconnectionwithTheta}
that we have isomorphisms $F_M \to F_P$ and $F_{(M')^*} \to F_{P'}$, and since $M^* = P' \del Y$ and $M' = P \del X$ it follows from \cite[Proposition 4.9]{Baker-Lorscheid25} that we have homomorphisms $F_{M^*} \to F_{P'}$ and $F_{M'} \to F_P$.
Hence, we have the following diagram of homomorphisms of pastures:
\renewcommand{\theequation}{\alph{equation}}
\setcounter{equation}{0}
\begin{equation} \label{eq:long diagram}
F_M \overset{\cong}{\to} F_{M^*} \to F_{P'} \overset{\cong}{\to} F_{(M')^*} \overset{\cong}{\to} F_{M'} \to F_P \overset{\cong}{\to} F_M.
\end{equation}
Here, the maps $F_M \to F_{M^*}$ and $F_{(M')^*} \to F_{M'}$ are the natural isomorphisms given by \cite[Proposition 4.8]{Baker-Lorscheid25},
and the maps $F_{P'} \to F_{(M')^*}$ and $F_P \to F_M$ are the inverses of the isomorphisms $F_M \to F_P$ and $F_{(M')^*} \to F_{P'}$.

By \autoref{lem:cross ratio surjection}, the homomorphisms $F_{M^*} \to F_{P'}$ and $F_{M'} \to F_P$ restrict to surjective homomorphisms of multiplicative groups.
It follows that the composition of the maps in \eqref{eq:long diagram} induces a surjection of multiplicative groups.
By \autoref{lem:pasture surjection}, we conclude that the composite map is an isomorphism, which means that all the intermediate maps must be isomorphisms as well. In particular, $F_{M'} \cong F_P$. 
On the other hand, we know from \autoref{thm:parallelconnectionwithTheta} that $F_P \cong F_M$, and thus $F_{M'} \cong F_M$ as desired.
\end{proof}

We have the following corollary in the case that $n = 3$.

\begin{thm}
Let $M$ be a matroid and let $T \subseteq E(M)$ be a coindependent triangle.
Then the foundation of the Delta-Wye exchange of $M$ along $T$ is isomorphic to the foundation of $M$.
\end{thm}

\begin{rem}
Note that if we replace the foundation by the universal pasture in the statement of \autoref{thm: segment-cosegment exchange}, the result remains true. This follows formally from Corollary 7.14 and Remark 7.15 of \cite{Baker-Lorscheid21b} upon noting that there is a bijection between connected components of $M$ and connected components of the segment-cosegment exchange of $M$ along $X$; see \autoref{lem: components of a segment-cosegment exchange} below for a straightforward proof of this fact.
\end{rem}

\begin{lemma} \label{lem: components of a segment-cosegment exchange}
If $M$ is a matroid with $X \subseteq E(M)$ so that $X$ is coindependent and $M|X \cong U_{2,n}$ for some $n \ge 2$, then there is a bijection between the connected components of $M$ and the connected components of the segment-cosegment exchange $P_X(M, \Theta_n) \del X$.
\end{lemma}
\begin{proof}
If $n = 2$, then $M$ and $P_X(M, \Theta_n) \del X$ are isomorphic because $\{x_i, y_i\}$ is a series pair for $i = 1,2$, so we may assume that $n \ge 3$.
If $M$ is connected, then $P_X(M, \Theta_n) \del X$ is connected by \cite[pg. 456, Ex. 6]{Oxley06} and the result follows, so we may assume that $M$ is disconnected.
Since $n \ge 3$ we know that $M|X$ is connected, and therefore $X$ is contained some component of $M$.
So $M = M_1 \oplus M_2$ where $M_1$ is connected and $X \subseteq E(M_1)$ (and $M_2$ may or may not be connected).

We will first show that $P_X(M, \Theta_n) = P_X(M_1, \Theta_n) \oplus M_2$.
Let $E$, $E_1$, and $E_2$ be the ground sets of $M$, $M_1$, and $M_2$, respectively.
For a matroid $N$ we write $\cF(N)$ for the set of flats of $N$.
Then
\begin{align*}
\cF(P_X(M, \Theta_n)) &= \{F \subseteq E  \cup Y \mid \textrm{$F \cap E \in \cF(M)$ and $F \cap (X \cup Y) \in \cF(\Theta_n)$}\} \\
&= \{F \subseteq E  \cup Y \mid \textrm{$F \cap E_i \in \cF(M_i)$ for $i = 1,2$ and $F \cap (X \cup Y) \in \cF(\Theta_n)$}\}\\
&= \{F \subseteq E  \cup Y \mid \textrm{$F \cap (E_1 \cup X \cup Y) \in \cF(P_X(M_1, \Theta_n))$ and $F \cap E_2 \in \cF(M_2)$}\}\\
&= \cF(P_X(M_1, \Theta_n) \oplus M_2).
\end{align*}
Here, the first and third lines follow from the definition of generalized parallel connection, and the second and fourth lines follow from the characterization of flats of a direct sum \cite[Proposition 4.2.16]{Oxley06}.
Therefore $P_X(M, \Theta_n) = P_X(M_1, \Theta_n) \oplus M_2$, and it follows from \cite[Proposition 4.2.19]{Oxley06} that $P_X(M, \Theta_n) \del X = (P_X(M_1, \Theta_n)\del X) \oplus M_2$.
Since $P_X(M_1, \Theta_n) \del X$ is connected by \cite[pg. 456, Ex. 6]{Oxley06}, it follows that the components of $P_X(M, \Theta_n) \del X$ are precisely $(E_1 - X) \cup Y$ and the components of $M_2$.
This gives a bijection between the components of $M$ and the components of $P_X(M, \Theta_n) \del X$ in which $E_1$ maps to $(E_1 - X) \cup Y$ and every other component of $M$ maps to itself.
\end{proof}

We turn to the proof of Corollary~\ref{corD} from the Introduction, whose statement we now recall:

\begin{cor}\label{cor: foundations for segment-cosegment exchange}
    Let $P$ be a pasture, and let $M$ be an excluded minor for representability over $P$. Then every segment-cosegment exchange of $M$ is also an excluded minor for representability over $P$.
\end{cor}

\begin{proof}
    Let $M$ be an excluded minor for $P$-representability, so $M$ is not $P$-representable, but every proper minor of $M$ is $P$-representable.
    In particular, it follows from \cite[Lemma 4.10]{Baker-Lorscheid25} that $M$ is simple and cosimple.
    Let $M|X \cong U_{2,n}$ for some $n \ge 2$ so that $X$ is coindependent in $M$, and let $M'$ be the segment-cosegment exchange of $M$ on $X$.
    It follows from \autoref{thm: segment-cosegment exchange} that $M'$ is not $P$-representable, so it suffices to show that every proper minor of $M'$ is $P$-representable.
    If $n = 2$, then $M' \cong M$ and the result holds, so we may assume that $n \ge 3$.
    Let $e \in E(M')$.
    We consider two cases.
    First suppose that $e = y_i$ for some $i \in [n]$.
    By \cite[Lemma 2.13]{Oxley-Semple-Vertigan} we know that $M'/y_i$ is isomorphic to the segment-cosegment exchange of $M \backslash x_i$ along $X - x_i$.
    Since $M \backslash x_i$ is $P$-representable, it follows from \autoref{thm: segment-cosegment exchange} that $M'/y_i$ is also $P$-representable.
    In $M' \backslash y_i$, the set $Y - y_i$ is contained in a series class because $M' |Y \cong U_{2,n}$.
    By \cite[Lemma 4.10]{Baker-Lorscheid25}, the cosimplification of $M' \backslash y_i$ has foundation isomorphic to the foundation of $M' \backslash y_i$.
    Since the cosimplification of $M' \backslash y_i$ is a minor of $M'/y_j$ for some $j \ne i$, it follows that $M' \backslash y_i$ is $P$-representable.

    Next suppose that $e \notin Y$.
    Then $M' \backslash e = P_X(M \backslash e, \Theta_n)\backslash X$ by \cite[Proposition~11.4.14~(iv)]{Oxley06}, and since $M \backslash e$ is $P$-representable it follows from \autoref{thm: segment-cosegment exchange} that $M' \backslash e$ is $P$-representable.
    It remains to show that $M' /e$ is $P$-representable.
    If $e$ is not spanned by $X$ in $M$, then by \cite[Lemma 2.16]{Oxley-Semple-Vertigan} we know that $M'/e$ is isomorphic to the segment-cosegment exchange of $M/e$ along $X$, and it follows from \autoref{thm: segment-cosegment exchange} that $M'/e$ is $P$-representable.
    So we may assume that $e$ is spanned by $X$ in $M$.
    Then $M|(X \cup e) \cong U_{2, n+1}$ because $M$ is simple, so $U_{2, n+1}$ is $P$-representable, and therefore $U_{n-1, n+1}$ is $P$-representable by \cite[Proposition 4.8]{Baker-Lorscheid25}.
    By \cite[Lemma 2.15]{Oxley-Semple-Vertigan} we know that $M'/e$ is isomorphic to the 2-sum of $M/e \backslash (X - x_i)$ and a copy of $U_{n-1, n+1}$ for some $i \in [n]$.
    Since both of these matroids are $P$-representable, it follows from \autoref{thmB} that $M'/e$ is $P$-representable.
\end{proof}

\subsection{Application to a conjecture by Pendavingh and van Zwam}
\label{sec:Pvz}

In this final section, we turn to the proof of Corollary~\ref{corC}. As preparation, we recall that the universal partial field $\P_M$ of a representable matroid $M$ is determined by its foundation $F_M$. 

According to \cite[Lemma 2.14]{Baker-Lorscheid25b}, for every pasture $P$ that maps to some partial field $F$, there is a universal map $\pi_P:P\to\Pi P$ to a partial field $\Pi P$ such every other map $f:P\to F$ to a partial field $F$ factors uniquely through $\pi_P$. 

The partial field $\Pi P$ is defined as follows: let $I$ be the ideal of the group ring $\Z[P^\times]$ which is generated by all terms $a+b+c$ that appear in the null set $N_P$. Then $\Pi P$ is the partial field $(P^\times, \Z[P^\times]/I)$; as a pasture, it can be described as
\[
 \Pi P \ = \ \past{P}{\gen{a+b+c\mid a+b+c\in I}}.
\]
The pasture morphism $\pi_P:P\to\Pi P$ is the quotient map. Note that since $P$ maps to some partial field, $I$ is a proper ideal of $\Z[P^\times]$ and thus $\Pi P$ is indeed a partial field (since $1\neq0$).

If $P=F_M$ is the foundation of a representable matroid $M$, its universal partial field is $\P_M=\Pi F_M$. This follows at once from a comparison of the universal properties of $\Pi F_M$ and $\P_M$: either of these partial fields represents the functor that associates with a partial field $F$ the set of rescaling classes of $M$ over $F$.

\begin{cor}\label{cor: universal partial field for segment-cosegment exchange}
 Let $M$ be a matroid, let $X \subseteq E(M)$ so that $X$ is coindependent and $M|X \cong U_{2,n}$ for some $n \ge 2$, and assume that $M$ is representable over some partial field. Then the universal partial field of the segment-cosegment exchange of $M$ along $X$ is isomorphic to the universal partial field of $M$.
\end{cor}

\begin{proof}
 Let $M'$ be the segment-cosegment exchange of $M$ along $X$. Let $F_M$ and $F_{M'}$ be the foundations of $M$ and $M'$, respectively. By \autoref{thm: segment-cosegment exchange}, $F_{M'}\simeq F_M$, which implies
 \[
  \P_{M'} \ = \ \Pi F_{M'} \ \simeq \ \Pi F_M \ = \ \P_M,
 \]
 since the functor $\Pi$ preserves isomorphisms.
\end{proof}

\begin{small}
 \bibliographystyle{plain}
 \bibliography{matroid}
\end{small}

\end{document}